\documentclass[12pt]{article}

\usepackage{color}
\usepackage{mathrsfs}

\usepackage{amssymb, amsthm, amsfonts, amsxtra, amsmath}
\usepackage{latexsym}
\usepackage{mathabx}
\usepackage{bbm} 

\usepackage[all]{xy}
\usepackage{graphics}
\usepackage{latexsym}
\usepackage{makeidx}
\usepackage{rotating}

\theoremstyle{plain}

\newtheorem{Theorem}{Theorem}[section]
\newtheorem{Lem}[Theorem]{Lemma}
\newtheorem{Prop}[Theorem]{Proposition}
\newtheorem{Cor}[Theorem]{Corollary}
\theoremstyle{definition}
\newtheorem{Def}[Theorem]{Definition}
\newtheorem{Not}[Theorem]{Notation}
\newtheorem{Exa}[Theorem]{Example}
\newtheorem{Exas}[Theorem]{Examples}
\newtheorem{Rem}[Theorem]{Remark}
\newtheorem{Rems}[Theorem]{Remarks}

\date{}

\DeclareMathOperator{\Mor}{Mor}
\DeclareMathOperator{\End}{End}
\DeclareMathOperator{\Hom}{Hom}
\DeclareMathOperator{\Rep}{Rep}

\DeclareMathOperator{\Fun}{Fun}
\DeclareMathOperator{\alg}{Alg}
\DeclareMathOperator{\comp}{Comp}
\DeclareMathOperator{\cC}{C}
\DeclareMathOperator{\PW}{P}

\DeclareMathOperator{\triv}{triv}
\DeclareMathOperator{\id}{id}

\newcommand{\SU}{\mathit{SU}} 
\newcommand{\G}{\mathbb{G}}
\newcommand{\X}{\mathbb{X}}
\newcommand{\Y}{\mathbb{Y}}
\newcommand{\C}{\mathbb{C}}

\newcommand{\aA}{\mathscr{A}}

\newcommand{\Hsp}{\mathscr{H}}
\newcommand{\ncl}{\mathrm{n\mhyph cl}}
\newcommand{\Hf}{\cat{H}_\textrm{f}}
\newcommand{\cat}[1]{\mathcal{#1}}
\newcommand{\Hmod}[1]{\mathcal{#1}}
\newcommand{\Co}[1]{C(#1)}
\newcommand{\CoL}[1]{C(#1)} 
\mathchardef\mhyph="2D

\newcommand{\norm}[1]{\left \| #1 \right \|}
\newcommand{\ConPos}{(P)} 

\newcommand{\Circt}{\mathop{\ooalign{$\ovoid$\cr\hidewidth\raise-.05ex\hbox{$\scriptstyle\mathsf T\mkern3.5mu$}\cr}}} 
\newcommand{\smCirct}{\mathop{\ooalign{$\scriptstyle\ovoid$\cr\hidewidth\raise-.05ex\hbox{$\scriptscriptstyle\mathsf T\mkern2.8mu$}\cr}}}  

\usepackage{tensor}
\newcommand{\lu}[2]{\tensor[^#1]{#2}{}}
\newcommand{\ld}[2]{\tensor[_#1]{#2}{}}

\usepackage{parskip} 
\makeatletter 
\def\thm@space@setup{%
  \thm@preskip=\parskip \thm@postskip=0pt
}
\makeatother

\usepackage{caption} 
\usepackage{subcaption} 

\numberwithin{equation}{section}

\newcommand{\texorpdfstring}[2]{#1}

\begin{document}

\hyphenation{Wo-ro-no-wicz}

\title{Tannaka--Kre\u{\i}n duality for compact quantum homogeneous spaces. I. General theory}
\author{K. De Commer\\
\small Department of Mathematics, University of Cergy-Pontoise,\\
\small UMR CNRS 8088, F-95000 Cergy-Pontoise, France\\
\small e-mail: Kenny.De-Commer@u-cergy.fr \\
\and M. Yamashita\\
\small Department of Mathematics, Ochanomizu University\\
\small Otsuka 2-1-1, Bunkyo, 112-8610, Tokyo, Japan\\
\small e-mail: yamashita.makoto@ocha.ac.jp}
\maketitle

\begin{abstract}
\noindent An ergodic action of a compact quantum group $\G$ on an operator algebra $A$ can be interpreted as a quantum homogeneous space for $\G$. Such an action gives rise to the category of finite equivariant Hilbert modules over $A$, which has a module structure over the tensor category $\Rep(\G)$ of finite-dimensional representations of $\G$.  We show that there is a one-to-one correspondence between the quantum $\G$-homogeneous spaces up to equivariant Morita equivalence, and indecomposable module $\cC^*$-categories over $\Rep(\G)$ up to natural equivalence. This gives a global approach to the duality theory for ergodic actions as developed by C. Pinzari and J. Roberts.
\end{abstract}

\emph{Keywords}: compact quantum groups; $\cC^*$-algebras; Hilbert modules; ergodic actions; module categories

AMS 2010 \emph{Mathematics subject classification}: 17B37; 20G42; 46L08

\section*{Introduction}

In the study of compact group actions on topological spaces, homogeneous spaces play a key r\^{o}le as fundamental building blocks. Ever since the foundational works of I.~Gelfand and M.~Neumark, the notion of unital $\cC^*$-algebras is known to be a rich generalization of compact topological spaces, and one frequently interprets them as function algebras on (compact) `quantum spaces'. In this more general noncommutative framework, a generally accepted notion of `compact quantum homogeneous space' for a compact group is that of a continuous ergodic action of the group on a unital $\cC^*$-algebra, that is, an action for which the scalars are the only invariant elements.

In the same way as compact topological spaces are generalized to unital $\cC^*$-algebras, S.L.~Woronowicz~\cite{Wor1,Wor2} generalized the notion of compact topological groups to that of \emph{compact quantum groups}. His axiom system for compact quantum groups is a very simple and natural one involving the coproduct homomorphism dualizing the product map of groups.  The resulting theory turns out to be strikingly rich, but at the same time as structured as the classical one.  As in the classical case, we have the Haar measure, the Peter--Weyl theory and the Tannaka--Kre\u{\i}n duality (\cite{Wor2,Wor4,Joy1}).

One may also formulate the notion of actions of compact quantum groups on quantum spaces, in a way which respects the Gelfand--Neumark duality when applied to the continuous map $G \times X \rightarrow X$ defining a classical group action. In this framework there is also a natural candidate for the `quantum homogeneous spaces' over compact quantum groups, by using the formalism of ergodic (co)actions~\cite{Pod1,Boc1}. In this paper, we aim to characterize such quantum homogeneous spaces in the spirit of the Tannaka--Kre\u{\i}n duality.

Such a duality theory for ergodic actions was already developed in~\cite{Pin1}, where the notion of quasi-tensor functor, a special kind of isometrically lax functor, was used. For practical purposes however, the lack of a strong tensor structure on such a functor makes it difficult to let algebra run its course in computations, due to the appearance of extraneous projections as stumbling blocks. Taking a cue from the theory of fusion categories, we rather formulate a duality theory in terms of \emph{module $\cC^*$-categories} over the tensor $\cC^*$-category of finite-dimensional representations of $\G$. Indeed, module categories over fusion categories are known to correspond to a good generalized notion of subgroup/homogeneous space (see A. Ocneanu's pioneering work in the subfactor context~\cite{Ocn1}, and more recent developments in the purely algebraic framework~\cite{And1,Ost1,Gel1}).

Module $\cC^*$-categories can equivalently, and more concretely, be described in terms of \emph{tensor functors into a category of bi-graded Hilbert spaces}. This formulation then makes at the same time the connection with the `fiber functor theory' from~\cite{BDV1}, which corresponds to non-graded Hilbert spaces and ergodic actions of \emph{full quantum multiplicity}, and with the theory of~\cite{Pin1}, which corresponds to considering one particular component of such a graded tensor functor. In the purely algebraic setting, such bi-graded tensor functors also lead to the construction of weak Hopf algebras, i.e.~quantum groupoids~\cite{Har1,Hay1,Eti1}, and Hopf--Galois actions~\cite{Ulb1,Ulb2,Sch1}. The relation with ergodic actions comes by means of a crossed product construction and a Morita theory for quantum groupoids, but we will not further go in to this in this paper. We also mention that a different kind of Tannaka-Kre\u{\i}n duality for classical homogeneous spaces was developed in \cite{Sug1}, and for actions on finite quantum spaces in~\cite{Ban1,Ban2} within the framework of planar algebras.

Here is a short summary of the contents of the paper. The first two sections will cover preliminaries and fix notations. They are meant as an aid for readers who are not familiar with the methodology. In the \emph{first section}, we will recall the basic concepts concerning compact quantum groups and quantum homogeneous spaces. In the \emph{second section}, we introduce the necessary prerequisites concerning $\cC^*$-categories, tensor $\cC^*$-categories and module $\cC^*$-categories.  Then, in the next five sections, we prove our main results. In the \emph{third section}, we explain how quantum homogeneous spaces lead to indecomposable module $\cC^*$-categories. In the \emph{fourth section}, we briefly expand on the algebraic content of a general compact quantum group action, so that in the \emph{fifth section}, we can concentrate on the essential part of the reconstruction of a quantum homogeneous space from an indecomposable module $\cC^*$-category. In the short \emph{sixth section} we show that this establishes essentially an equivalence between the two notions. In the \emph{seventh section}, we give further comments on the functoriality of this correspondence. In the \emph{appendix}, we explain the link between module $\cC^*$-categories and bi-graded tensor functors. It is mainly meant to provide details for, as well as to generalize, the remark which appears in the proof of Theorem 2.5 of~\cite{Eti1}.

In the accompanying paper~\cite{DeC3}, we apply the results of the present paper to the case of the compact quantum group $\SU_q(2)$.

\paragraph{Conventions} To have consistency when working with Hilbert $\cC^*$-modules, we will always take the inner product $\langle \xi, \eta \rangle$ of a Hilbert space to be linear in $\eta$ and antilinear in $\xi$.  When $\xi$ and $\eta$ are vectors in a Hilbert space $\Hsp$, we write $\omega_{\xi,\eta}$ for the functional $T \mapsto \langle \xi, T \eta\rangle$ on $B(\Hsp)$. When $A$ and $B$ are $\cC^*$-algebras, $A \otimes B$ denotes their minimal tensor product unless otherwise stated.

\section{Compact quantum groups and related structures}
\label{SecQG}

\subsection{Compact quantum groups}
\label{SubSecQG3}

\begin{Def}[\cite{Wor2}] A \emph{compact quantum group $\G$} consists of a unital $\cC^*$-algebra $\Co{\G}$ and a faithful unital $^*$-homomorphism $\Delta\colon \Co{\G}\rightarrow \Co{\G}\otimes \Co{\G}$ satisfying the coassociativity condition $(\Delta \otimes \id) \circ \Delta = (\id \otimes \Delta) \circ \Delta$ and the cancelation condition
\[
\lbrack \Delta(\Co{\G})(1\otimes \Co{\G})\rbrack^{\mathrm{\ncl}} = \Co{\G}\otimes \Co{\G}= \lbrack \Delta(\Co{\G})(\Co{\G}\otimes 1)\rbrack^{\mathrm{\ncl}},
\] where $\ncl$ means taking the norm-closed linear span.
\end{Def}

We recall from~\cite{Wor2} that any compact quantum group admits a unique positive state $\varphi_{\G}$ which satisfies
\begin{equation}\label{EqInvState}
(\id\otimes \varphi_{\G})(\Delta(x))= \varphi_{\G}(x)1= (\varphi_{\G}\otimes \id)(\Delta(x)),\qquad x\in \Co{\G}.
\end{equation}
This state is called the \emph{invariant state} (or the \emph{Haar state}) of $\Co{\G}$.

\begin{Def} The compact quantum group $\G$ is called \emph{reduced} if the invariant state $\varphi_{\G}$ is faithful.\end{Def}

In the rest of the paper, we will always work with reduced compact quantum groups. This is no serious restriction, as to any $\G$ one can associate a reduced companion which has precisely the same representation theory as $\G$.

\begin{Def}\label{DefCQGRep}
A \emph{unitary corepresentation} $u$ of $\Co{\G}$ on a Hilbert space $\Hsp_u$ is given by a unitary element $u$ of $B(\Hsp_u) \otimes \Co{\G}$ satisfying the multiplicativity condition
\[
(\id \otimes \Delta)(u) = u_{1 2} u_{1 3} \in B(\Hsp_u) \otimes \Co{\G}\otimes \Co{\G},
\]
where the leg numbering indicates at which slot in a multiple tensor product one places the element, filling the blank spots with units.  A unitary corepresentation $u$ is said to be finite-dimensional when $\Hsp_u$ is so.

When $u$ and $v$ are unitary corepresentations of $\Co{\G}$, an operator $T \in B(\Hsp_u, \Hsp_v)$ is said to be an \emph{intertwiner} between $u$ and $v$ if it satisfies $v (T \otimes 1) = (T \otimes 1) u$. A unitary corepresentation $u$ is called \emph{irreducible} if the space of intertwiners from $u$ to itself is one-dimensional.
\end{Def}

In what follows we will refer to unitary corepresentations of $\Co{\G}$ as unitary representations of $\G$.

\subsection{Quantum homogeneous spaces}
\label{SubQHom}

\begin{Def}[\cite{Boc1,Pod1}] Let $\G$ be a compact quantum group. An \emph{action} of $\G$ on a unital $\cC^*$-algebra $A$ is a faithful unital $^*$-homomorphism
\[
\alpha\colon A\rightarrow A\otimes \Co{\G}
\]
satisfying the coaction condition $(\id\otimes \Delta) \circ \alpha = (\alpha\otimes \id) \circ \alpha$ and the density condition
\[
\lbrack (1\otimes \Co{\G})\alpha(A)\rbrack^{\mathrm{\ncl}} = A\otimes \Co{\G}.
\]

We call the action \emph{ergodic} if the space
\[
A^{\G} = \{x\in A\mid \alpha(x)= x\otimes 1\}
\]
is  equal to $\mathbb{C}1$.  If $(A,\alpha)$ is an ergodic action, we will use the notation $A= \Co{\X}$, and refer to the symbol $\X$ as the quantum homogeneous space.
\end{Def}

If $\X$ is a quantum homogeneous space for $\G$, then $\Co{\X}$ carries a canonical faithful positive state $\varphi_{\X}$, determined by the identity
\[
(\id\otimes \varphi_{\G})(\alpha(x)) = \varphi_{\X}(x)1\qquad (x\in \Co{\X}).
\]
It is the unique state on $\Co\X$ which is $\alpha$-invariant, $(\varphi_{\X}\otimes \id)\alpha(x) = \varphi_{\X}(x)1$ for all $x\in \Co{\X}$.

\section{$\cC^*$-categories}\label{SecCCat}

\subsection{Semi-simple $\cC^*$-categories}

\begin{Def}[\cite{Ghe1}]
A $\cC^*$-\emph{category} $\cat{D}$ is a $\mathbb{C}$-linear category whose morphism spaces are Banach spaces satisfying the submultiplicativity condition $\|ST\| \leq \|S\|\|T\|$ for composition of morphisms $S$ and $T$, and admitting antilinear `involutions'
\[
^*\colon\Mor(X, Y) \rightarrow \Mor(Y, X), \quad T \mapsto T^*,
\]
which behave contravariantly and satisfy the $\cC^*$-condition $\|T^*T\| = \|T\|^2$ for each morphism $T$.  A linear functor between two $\cC^*$-categories is called a $\cC^*$-\emph{functor} if it preserves the $^*$-operation.
\end{Def}

\begin{Rem}\label{RemFuncCat}
Let $\cat{D}$ and $\cat{D}'$ be $\cC^*$-categories.  Let $\Fun(\cat{D}, \cat{D}')$ be the category
\begin{itemize}
\item
whose objects are the $\cC^*$-functors from $\cat{D}$ to $\cat{D}'$, and
\item
whose morphisms between two functors $F, G\colon \cat{D} \rightarrow \cat{D}'$ consist of the natural transformations $\phi_\bullet = (\phi_X \colon F X \rightarrow G X)_{X \in \cat{D}}$ such that $(\norm{\phi_X})_{X \in \cat{D}}$ is uniformly bounded.
\end{itemize}
Then $\Fun(\cat{D}, \cat{D}')$ is a $\cC^*$-category with the norm $\norm{\phi_\bullet} = \sup_{X \in \cat{D}} \norm{\phi_X}$ and the involution $(\phi^*)_X = (\phi_X)^*$.
\end{Rem}

\begin{Def}[\cite{Ghe1}]
We say that an object $X$ in a $\cC^*$-category $\cat{D}$ is \emph{simple} if $\Mor(X, X)$ is isomorphic to $\C$.  We call $\cat{D}$ \emph{semi-simple}~\cite[Section 1.6]{Mug2} if $\cat{D}$ admits finite direct sums and if any of its objects is isomorphic to a finite direct sum of simple objects.
\end{Def}

\begin{Rem}\label{RemSSFinDimMor}
A $\cC^*$-category $\cat{D}$ is semi-simple if and only if all morphism spaces are finite-dimensional and `idempotents split'. The latter condition means that any self-adjoint projection $p\in \Mor(X,X)$ is of the form $v v^*$ for some isometry $v\in \Mor(Y,X)$. Furthermore, a semi-simple $\cC^*$-category also has a zero object $0$, i.e. an object which is both initial and terminal.
\end{Rem}

\begin{Def}
Let $J$ be a set, and $\cat{D}$ a semi-simple $\cC^*$-category.  We say that $\cat{D}$ is \emph{based on} $J$ if we are given a bijection between $J$ and a maximal family of mutually non-isomorphic simple objects in $\cat{D}$. We then write $X_r$ for the simple object associated with $r\in J$.
\end{Def}

By definition, any object $X$ in a semi-simple $\cC^*$-category $\cat{D}$ based on $J$ is isomorphic to a direct sum $\oplus_{r\in J} m_r X_r$. The integer $m_r$ is called the \emph{multiplicity} of $X_r$ in $X$, and is uniquely determined by $m_r= \dim(\Mor(X_r,X))$. Then for any object $X$ and any irreducible $X_r$, the complex vector space $\Mor(X_r, X)$ admits a natural structure of Hilbert space by the inner product $\langle S, T \rangle = S^* T \in \Mor(X_r, X_r) = \mathbb{C}$.

Examples of semi-simple $\cC^*$-categories will be presented in Section~\ref{SecTK1} and the appendix. They can be seen as categorified versions of Hilbert spaces, cf. the slightly different context of~\cite{Bae1}.  As with Hilbert spaces, there is essentially only one semi-simple $\cC^*$-category for each cardinal number, the cardinality of the set of isomorphism classes of irreducible objects in the given semi-simple $\cC^*$-category, cf. Lemma~\ref{LemEquii}. However, true to this analogy, they arise in various presentations in practical situations, from concrete to abstract. For the moment, it will suffice to have the following characterization of equivalences between semi-simple $\cC^*$-categories.

\begin{Lem}\label{LemEquiii}
Let $\cat{D}$ and $\cat{D}'$ be semi-simple $\cC^*$-categories, with $\cat{D}$ based on an index set $J$. Let $F$ be a $\cC^*$-functor from $\cat{D}$ to $\cat{D}'$. Then $F$ is an equivalence of categories if and only if the set $\{F(X_r)\mid r\in J\}$ forms a maximal set of mutually non-isomorphic irreducible objects in $\cat{D}'$.
\end{Lem}

\begin{proof}
The necessity of the condition is obvious.  Let us see that it is also sufficient.  Let $X$ be an irreducible object of $\cat{D}$ and let $m$ be a nonnegative integer. Then the $\cC^*$-algebra $\End(m X)$ is isomorphic to $M_m(\mathbb{C})$, where the identity morphisms of the direct summands form a partition of unity by mutually equivalent minimal projections.  Since $F(X)$ is also an irreducible object, it follows that $F$ induces a $\cC^*$-algebra isomorphism between $\End(m X)$ and $\End(F(m X)) \cong \End(m F(X))$.  More generally, given a finite direct sum $X=\oplus_{r\in J} m_r X_r$, we can conclude that $F$ provides an isomorphism between $\End(X)$ and $\End(F(X))$. Finally, by considering this argument for $X\oplus Y$, we conclude that $F$ gives a bijection from $\Mor(X,Y)$ to $\Mor(F(X),F(Y))$ for any objects $X,Y$, that is, $F$ is a fully faithful functor.

As the set $\{F(X_r)\mid r\in J\}$ forms a maximal set of mutually non-isomorphic irreducible objects in $\cat{D}'$, we also have that $F$ is essentially surjective. From~\cite[Theorem IV.4.1]{Mac1}, we conclude that $F$ is an equivalence.
\end{proof}

\subsection{Tensor $\cC^*$-categories}

\begin{Def}\cite{Dop1}
A \emph{(strict) tensor $\cC^*$-category} $\cat{C}=(\cat{C},\otimes,\mathbbm{1})$ consists of a $\cC^*$-category $\cat{C}$ together with a bilinear $C^*$-functor $\otimes\colon \cat{C}\times \cat{C}\rightarrow \cat{C}$ and an object $\mathbbm{1} \in \cat{C}$ such that there are equalities of functors
\begin{align*}
-\otimes (-\otimes -) &=(-\otimes-)\otimes - ,&  \mathbbm{1}\otimes - &= \id_{\cat{C}}= -\otimes \mathbbm{1}.
\end{align*}
\end{Def}

The `strictness' condition refers to the on the nose associativity of $\otimes$. In most examples which arise in practice, the associativity only holds up to certain coherence isomorphisms~\cite[Chapter VII]{Mac1}. But for the cases we will encounter, the coherence isomorphisms will be obvious and one can safely ignore them. Also for abstract tensor categories, one can almost always restrict oneself to the setting of strict tensor categories by Mac Lane's coherence theorem~\cite[Section VII.2]{Mac1}. This coherence result holds as well on the $\cC^*$-level.

\begin{Def}[\cite{Dop1,Lon1}]\label{DefTensor}
Let $\cat{C}$ be a tensor $\cC^*$-category. An object $U$ in $\mathcal{C}$ is said to \emph{admit a conjugate} or \emph{dual} if there exists a triple $(\bar{U},R_U,\bar{R}_U)$ with $\bar{U} \in \cat{C}$ and $(R_U,\bar{R}_U)$ a couple of morphisms
\begin{align*}
R_U\colon \mathbbm{1} &\rightarrow \bar{U}\otimes U, &
\bar{R}_U\colon \mathbbm{1} &\rightarrow U\otimes \bar{U}
\end{align*}
satisfying the \emph{conjugate equations}
\begin{align}\label{EqDualityMors}
(\bar{R}_U^*\otimes \id_U) (\id_U\otimes R_U) &= \id_U, &
(R_U^*\otimes \id_{\bar{U}}) (\id_{\bar{U}}\otimes \bar{R}_U) &= \id_{\bar{U}}.
\end{align}

The full subcategory of all objects in $\cat{C}$ admitting duals is denoted by $\mathcal{C}_\textrm{f}$. A tensor $\cC^*$-category $\mathcal{C}$ is called \emph{rigid} if $\mathcal{C} = \mathcal{C}_\textrm{f}$.
\end{Def}

\begin{Rems}
\begin{enumerate}
\item\cite[Theorem 2.4]{Lon1} When $U$ and $V$ are in $\cat{C}_\textrm{f}$, the product $\bar{V} \otimes \bar{U}$ of their duals is in duality with $U \otimes V$. Moreover, if $(\bar{U},R_U,\bar{R}_U)$ makes a dual for $U$, then $(U,\bar{R}_U,R_U)$ makes a dual for $\bar{U}$. It follows that $\cat{C}_\textrm{f}$ is a rigid $\cC^*$-tensor subcategory of $\cat{C}$.
\item For any $U$, the object $\bar{U}$, when it exists, is unique up to isomorphism. If $(R_U,\bar{R}_U)$ satisfy the conjugate equations, then for any $\lambda \in \mathbb{C}^\times$ also $(\lambda R_U,\bar{\lambda}^{-1}\bar{R}_U)$ satisfy the same equations. When the unit of $\cat{C}$ is irreducible, then for $U$ irreducible and $\bar{U}$ a fixed dual, this is the only arbitrariness in the choice of $(R_U,\bar{R}_U)$.
\item When the unit of $\cat{C}$ is irreducible, then for any irreducible $U$ with dual $\bar{U}$, one can always arrange for a solution $(R_U,\bar{R}_U)$ of the conjugate equations which is \emph{normalized}, i.e.~such that $R_U^*R_U = \bar{R}_U^*\bar{R}_U$. Then by the above scaling result, $\dim_q(U) = R_U^*R_U$ is a strictly positive real number which is uniquely determined by $U$. It is called the \emph{quantum dimension} of $U$.
\end{enumerate}
\end{Rems}

\begin{Exas}\label{ExaTenso}
\begin{enumerate}
\item
The category of all Hilbert spaces and bounded maps is a tensor $\cC^*$-category for the ordinary tensor product of Hilbert spaces. The maximal rigid subcategory consists of all finite-dimensional Hilbert spaces. If $\mathscr{H}$ is a finite-dimensional  Hilbert space, the complex conjugate space $\overline{\Hsp}$ can be taken as its conjugate object, where the maps $R_{\Hsp}$ and $\bar{R}_{\Hsp}$ are given by
\begin{align*}
R_{\Hsp}^*\colon \overline{\Hsp} \otimes \Hsp &\rightarrow \mathbb{C}, \quad \bar{\xi}\otimes \eta\rightarrow \langle \xi,\eta\rangle, &
R_{\Hsp}^*\colon \Hsp \otimes \overline{\Hsp} &\rightarrow \mathbb{C}, \quad \xi\otimes \bar{\eta}\rightarrow \langle \eta,\xi\rangle.
\end{align*}
\item
For any compact quantum group $\mathbb{G}$, the category $\Rep(\G)$ of its finite-dimensional unitary representations together with the intertwiners forms a rigid tensor $\cC^*$-category with irreducible unit object. The tensor product $u \Circt v$ of two representations $u$ and $v$ is defined to be the representation on $\Hsp_u\otimes \Hsp_v$ given by the unitary $u_{13}v_{23} \in B(\Hsp_u) \otimes B(\Hsp_v) \otimes \Co{\G}$.  When $u$ is an object of $\Rep(\G)$, its dual can be given by a unitarization of $(j \otimes \id)(u^{-1}) \in B(\overline{\Hsp_u}) \otimes \Co{\G}$, where $j \colon B(\Hsp_u) \rightarrow B(\overline{\Hsp_u})$ is the natural anti-isomorphism characterized by $j(T) \bar{\xi} = \overline{T^* \xi}$. Unlike the case of Hilbert spaces or compact groups, $u \Circt v$ is not isomorphic to $v\Circt u$ in general.

\item\cite{Dop1,Wor0} For a fixed $\cC^*$-category $\cat{D}$, let $\End(\cat{D})$ denote the category of $\cC^*$-endofunctors, cf. Remark \ref{RemFuncCat}. Then $\End(\cat{D})$ is a tensor $\cC^*$-category, with the $\otimes$-structure $F\otimes G = F\circ G$ given by the composition of endofunctors, and with the identity functor providing the unit. The associated rigid category $\End(\cat{D})_\textrm{f}$ consists of adjointable functors whose unit and co-unit maps are uniformly bounded.
\end{enumerate}
\end{Exas}

We recall the notion of strong tensor functor and tensor equivalence.

\begin{Def}
Let $\cat{C}_1$ and $\cat{C}_2$ be two tensor $\cC^*$-categories. A \emph{strong tensor $\cC^*$-functor} from $\cat{C}_1$ to $\cat{C}_2$ consists of a $\cC^*$-functor $F\colon \cat{C}_1\rightarrow \cat{C}_2$ together with natural unitary transformations
\begin{align*}
\psi_{U,V}\colon F(U)\otimes F(V) &\rightarrow F(U\otimes V),&
c\colon \mathbbm{1}_{\mathcal{C}_2} &\rightarrow F(\mathbbm{1}_{\mathcal{C}_1}),
\end{align*}
satisfying certain coherence conditions~\cite[Section 1.2]{Mug1}.

It is called a \emph{tensor equivalence} if the underlying functor $F$ is an equivalence.
\end{Def}

\begin{Exa}
If $\G$ is a compact quantum group, there is a natural \emph{forgetful functor} from $\Rep(\G)$ to $\Hf$, sending each unitary representation $u$ to the underlying Hilbert space $\mathscr{H}_u$, and acting as the identity on intertwiners. The natural transformations $\psi$ and $c$ are identity maps. In general, there can exist other faithful strong tensor C$^*$-functors from $\Rep(\G)$ to $\Hf$ besides this canonical one, cf.~\cite{BDV1}, but each one of them determines a unique compact quantum group (\cite{Wor4}).
\end{Exa}

The following lemma will be used at some point.

\begin{Lem}[\cite{Lon1}]\label{LemRig}
Let $\cat{C}_1$ and $\cat{C}_2$ be tensor $\cC^*$-categories, and $F\colon \cat{C}_1\rightarrow \cat{C}_2$ a strong tensor $\cC^*$-functor. If $\cat{C}_1$ is rigid, then the image of $F$ is contained in $(\cat{C}_2)_\textrm{f}$.
\end{Lem}

\begin{proof}
If $U\in \cat{C}_1$, then the compatibility of $F$ with the tensor products can be used to construct a duality between $F(U)$ and $F(\bar{U})$. Hence the image of $F$ is inside $(\cat{C}_2)_\textrm{f}$.
\end{proof}

\subsection{Module $\cC^*$-categories}

\begin{Def}\label{DefModCat}
Let $\cat{C}$ be a tensor $\cC^*$-category with unit object $\mathbbm{1}$, and $\cat{D}$ a $\cC^*$-category. One says that $\cat{D}=(\cat{D},M,\phi,e)$ is a \emph{$\cat{C}$-module $\cC^*$-category} if $M\colon \mathcal{C}\times \cat{D} \rightarrow \cat{D}$ is a bilinear $^*$-functor with natural unitary transformations
\begin{align*}
\phi\colon M((-\otimes -),-) &\overset{\sim}{\rightarrow} M(-,M(-,-)), &
e\colon M(\mathbbm{1},-) &\overset{\sim}{\rightarrow} \id,
\end{align*}
satisfying certain obvious coherence conditions, cf.~\cite{Ost2}, which we will spell out below.

We say that $\cat{D}$ is \emph{semi-simple} if the underlying $\cC^*$-category is semi-simple.

We say that $\cat{D}$ is \emph{indecomposable} or \emph{connected} if, for all non-zero $X,Y\in \cat{D}$, there exists an object $U\in \cat{C}$ such that $\Mor(M(U,Y),X)\neq 0$.
\end{Def}

In the following, we will use the more relaxed notation $U\otimes X$ for $M(U,X)$, and similarly for morphisms. The coherence conditions can then be written in the following form, as the commutation of the diagrams
\begin{equation}
\xymatrix@C=3.5em{
(U\otimes V \otimes W)\otimes X \ar[r]^{\phi_{U,V\otimes W,X}} \ar[d]_{\phi_{U\otimes V,W,X}} & U\otimes((V\otimes W)\otimes X)\ar[d]^{\id_U\otimes \phi_{V,W,X}} \\ (U\otimes V)\otimes (W\otimes X) \ar[r]^{\phi_{U,V,W\otimes X}}& U\otimes (V\otimes (W\otimes X)),
}
\end{equation}
and
\begin{equation}\label{EqCoh2}
\xymatrix{& U\otimes (\mathbbm{1}\otimes X)\ar[dr]^{\id_U\otimes e_X}& \\
U\otimes X \ar[ur]^{\phi_{U,\mathbbm{1},X}}\ar[dr]_{\phi_{\mathbbm{1},U,X}} \ar[rr]^{\id_{U\otimes X}} && U\otimes X.\\ & \ar[ur]_{e_{U\otimes X}}  \mathbbm{1}\otimes (U\otimes X)&
}
\end{equation}

\begin{Exas}\label{ExaQSubGrpEtc}
\begin{enumerate}
\item
Let $\cat{D}$ be a $\cC^*$-category. Then $\cat{D}$ is a module $\cC^*$-category for $\End(\cat{D})$ in an obvious way.
\item
Let $\G$ be a compact (quantum) group and $\mathbb{H}$ be a closed (quantum) subgroup of $\G$.  Then $\Rep(\mathbb{H})$ is a $\Rep(\G)$-module $\cC^*$-category in a natural way: the action of $\pi \in \Rep(\G)$ on $\theta \in \Rep(\mathbb{H})$ is defined as $\pi_{|\mathbb{H}}\otimes \theta$.  In other words, this is induced by the restriction functor $\Rep(\G) \rightarrow \Rep(\mathbb{H})$, which is a strong tensor $\cC^*$-functor.
\item
More generally, if $\cat{C}_1$ and $\cat{C}_2$ are tensor $\cC^*$-categories, and $F$ a strong tensor $\cC^*$-functor from $\cat{C}_1$ to $\cat{C}_2$, then $\cat{C}_2$ becomes a $\cat{C}_1$-module $\cC^*$-category by the association $M(X,Y) = F(X)\otimes Y$.
\end{enumerate}
\end{Exas}

We will need the following interplay between dual objects and the module structure.

\begin{Lem}\label{LemModCatFrob}
Let $\cat{C}$ be a rigid tensor $\cC^*$-category, and let $\cat{D}$ be a $\C$-module $\cC^*$-category.  For any $U$ in $\cat{C}$ and any objects $X, Y$ in $\cat{D}$, we have an isomorphism $\Mor(U \otimes Y, X) \cong \Mor(Y, \bar{U} \otimes X)$, called \emph{the Frobenius isomorphism} associated with $(R_U, \bar{R}_U)$.
\end{Lem}

\begin{proof}
This can be proved by a standard argument involving the conjugate equations, cf. Proposition \ref{PropCorrModTen}.
\end{proof}

The appropriate notion of morphisms between module $\cC^*$-categories is the following.

\begin{Def}\label{DefMorMod}
Let $\cat{D}$ and $\cat{D}'$ be module $\cC^*$-categories over a fixed tensor $\cC^*$-category $\cat{C}$.  A \emph{$\cat{C}$-module homomorphism} from $\cat{D}$ to $\cat{D}'$ is given by a pair $(G, \psi)$, where $G$ is a functor from $\cat{D}$ to $\cat{D}'$ and $\psi$ is a unitary natural equivalence $G (-\otimes -)\rightarrow -\otimes G-$, such that the diagrams of the form
\begin{equation}\label{EqCModHomFnctrTrivAct}
\xymatrix{
G(\mathbbm{1}\otimes X) \ar[d]_{G(e)} \ar[r]^{\psi_{\mathbbm{1},X}} & \mathbbm{1}\otimes G X \ar[dl]^{e} \\
G X &
}
\end{equation}
and
\begin{equation}\label{EqCModHomFnctr}
\xymatrix@C=4em{\\ &      U\otimes G(V\otimes X)    \ar[rd]^{\id_U\otimes \psi_{V, X}}     &\\
G(U\otimes (V\otimes X)) \ar[ru]^{\psi_{U,V\otimes X}} \ar[d]_{G(\phi_{U, V,X})} && U\otimes (V\otimes G X) \ar[d]^{\phi_{U, V, G X}} \\
G ((U \otimes V)\otimes X) \ar[rr]_{\psi_{U \otimes V, X}} && (U \otimes V)\otimes G X
}
\end{equation}
commute.

An \emph{equivalence} between $\cat{D}$ and $\cat{D}'$ is a morphism $(G,\psi)$ for which $G$ is an equivalence of categories.
\end{Def}

The following section is dedicated to the $\Rep(\G)$-module $\cC^*$-categories which are the star actors of this paper.

\section{Equivariant Hilbert modules}
\label{SecTK1}

\begin{Def}[\cite{Baa1}]
Let $\X$ be a quantum homogeneous space for a compact quantum group $\G$. An \emph{equivariant Hilbert $\cC^*$-module} $\Hmod{E}$ over $\X$ is a right Hilbert $\Co{\X}$-module $\Hmod{E}$, carrying a coaction $\alpha_{\Hmod{E}}\colon\Hmod{E}\rightarrow \Hmod{E}\otimes\Co{\G}$, where the right hand side is the exterior product of $\Hmod{E}$ with the standard right Hilbert $\Co{\G}$-module $\Co{\G}$, satisfying the density condition
\[
\lbrack (1\otimes\Co{\G})\alpha_{\Hmod{E}}(\Hmod{E})\rbrack^{\ncl} = \Hmod{E} \otimes \Co{\G} = \lbrack\alpha_{\Hmod{E}}(\Hmod{E})(1\otimes \Co{\G})\rbrack^{\ncl}
\]
and the compatibility conditions
\begin{enumerate}
\item $\forall x\in \Co{\X}, \forall \xi \in \Hmod{E}\colon\alpha_{\Hmod{E}}(\xi \cdot x)  = \alpha_{\Hmod{E}}(\xi)\alpha_{\X}(x)$,
\item $\forall \xi,\eta\in \Hmod{E}\colon \langle \alpha_{\Hmod{E}}(\xi),\alpha_{\Hmod{E}}(\eta)\rangle_{\Co{\X}\otimes \Co{\G}} = \alpha_{\X}(\langle \xi,\eta\rangle_{\Co{\X}})$.
\end{enumerate}
\end{Def}

\begin{Rem}\label{RemHilb}
An equivariant Hilbert $\cC^*$-module is necessarily saturated, and in particular faithful as a right $\Co{\X}$-module. Indeed, otherwise the closed linear span of $\{\langle \xi,\eta\rangle_{\Co{\X}}\mid \xi,\eta\in \Hmod{E}\}$ would give a proper equivariant closed 2-sided ideal $\mathcal{I}$ in $\Co{\X}$. But any invariant state on $\Co{\X}/\mathcal{I}$ would induce a non-faithful invariant state over $\Co{\X}$, which is a contradiction.
\end{Rem}

To any equivariant Hilbert $\Co{\X}$-module one can associate a special unitary which implements the coaction.

\begin{Def}\label{DefEqvrModUntryMorRep}
Let $\X$ be a quantum homogeneous space for a compact quantum group $\G$, and $\Hmod{E}$ an equivariant Hilbert $\cC^*$-module over $\X$. One defines the \emph{associated unitary morphism}
\[
X_{\Hmod{E}} \in \mathcal{L}_{\Co{\X} \otimes \Co{\G}} \left(\Hmod{E} \otimes_{\alpha_\X} (\Co{\X} \otimes \Co{\G}), \Hmod{E} \otimes \Co{\G} \right)
\]
by the formula $X_{\Hmod{E}}(\xi\otimes (x\otimes h)) = \alpha_{\Hmod{E}}(\xi)(x\otimes h)$.
\end{Def}

\begin{Exa}\label{ExaHilb}
Consider a set $\bullet$ with one element, and consider $\Co{\bullet}=\mathbb{C}$ with the trivial right action
\[
\alpha_{\triv}\colon\Co{\bullet}\rightarrow \Co{\bullet}\otimes \Co{\G}, \quad 1\rightarrow 1\otimes 1.
\]
Then an equivariant Hilbert $\cC^*$-module over $\bullet$ is nothing but a representation of $\G$.  Indeed, a right Hilbert $\Co{\bullet}$-module is just a Hilbert space $\mathscr{H}$. Then the receptacle of the unitary operator in Definition~\ref{DefEqvrModUntryMorRep} can be identified with $B(\mathscr{H})\otimes \Co{\G}$. This gives the correspondence of the equivariant Hilbert $\cC^*$-modules over $\bullet$ and the unitary representations of $\G$. We will denote the equivariant Hilbert space associated to $u$ as $(\Hsp_u,\delta_u)$.
\end{Exa}

We will be particularly interested in a subcategory of equivariant Hilbert $\cC^*$-modules which admit a nice decomposition into irreducible objects.

\begin{Def} An equivariant Hilbert $\cC^*$-module $\Hmod{E}$ is called \begin{itemize}
\item
\emph{finite} if it is finitely generated projective as a right $\Co{\X}$-module, and
\item
\emph{irreducible} if the space
\[
\mathcal{L}_{\G}(\Hmod{E}) = \{T \in \mathcal{L}(\Hmod{E})\mid \alpha_{\Hmod{E}}(T\xi) = (T \otimes 1)\alpha_{\Hmod{E}}(\xi) \textrm{ for all } \xi \in \Hmod{E}\}
\]
is one-dimensional.
\end{itemize}
\end{Def}

Any irreducible equivariant Hilbert $\cC^*$-module is finite in the above sense, as seen in the next proposition.

\begin{Prop}\label{PropFinn}
An equivariant $\cC^*$-module is finite if and only if the $\cC^*$-algebra $\mathcal{L}_{\G}(\Hmod{E})$ is finite-dimensional.
\end{Prop}

\begin{proof}
Let $X_{\Hmod{E}}$ be the unitary morphism associated with $\alpha_\xi$ as in Definition~\ref{DefEqvrModUntryMorRep}.  Then, the map $x \mapsto X_{\Hmod{E}} (x \otimes_{\alpha_\xi} 1) X_{\Hmod{E}}^*$ defines a coaction of $\Co{\G}$ on $\mathcal{L}(\Hmod{E})$, and the ideal of compact endomorphisms is a $\G$-invariant subalgebra~\cite{Baa1}.  Moreover, $\mathcal{L}_{\G}(\Hmod{E})$ is precisely the $\G$-fixed point subalgebra of $\mathcal{L}(\Hmod{E})$.

First, let us prove that an equivariant module over $\X$ is finitely generated projective over $\Co{\X}$ when $\mathcal{L}_{\G}(\Hmod{E})$ is finite-dimensional.  We can reduce it to the case of $\mathcal{L}_{\G}(\Hmod{E}) = \C$ by taking a decomposition associated with a partition of unity by minimal projections in $\mathcal{L}_{\G}(\Hmod{E})$.  Then, taking any non-zero positive compact endomorphism $x$ of $\Hmod{E}$, we see that $(\id \otimes \varphi_{\G})(X_\alpha (x \otimes_{\alpha_{\Hmod{E}}} 1) X_\alpha^*)$ is simultaneously compact and nonzero positive scalar in $\mathcal{L}(\Hmod{E})$.  Hence $\Hmod{E}$ is finitely generated projective over $\Co{\X}$~\cite[Lemma~6.5]{Kas1}.

Conversely, suppose that we are given a finitely generated projective $\Co{\X}$-module $\Hmod{E}$ admitting a compatible corepresentation of $\Co{\G}$.  Then, the crossed product module $\Hmod{E} \rtimes \G$, which is finitely generated projective over $\Co{\X} \rtimes \G$, admits a natural faithful representation of $\mathcal{L}_\G(\Hmod{E})$ as $\Co{\X} \rtimes \G$-module homomorphisms.

By the ergodicity of $\G$ on $\X$, we know that $\Co{\X} \rtimes \G$ is a direct sum of algebras of compact operators~\cite{Boc1}.  Hence, for any finitely generated projective module over $\Co{\X} \rtimes \G$, the module endomorphisms must form a finite-dimensional algebra.  This implies that $\mathcal{L}_\G(\Hmod{E})$ is finite-dimensional.
\end{proof}

In particular, any irreducible equivariant Hilbert $\cC^*$-module $\Hmod{E}$ over $\Co{\X}$ gives another quantum homogeneous space $\mathcal{L}(\Hmod{E}) = \mathcal{K}(\Hmod{E})$, by the action as given in the beginning of the above proof.

\begin{Def} A quantum homogeneous space $\mathbb{Y}$ is called \emph{equivariantly Morita equivalent} to $\X$ if there exists an irreducible equivariant Hilbert $\cC^*$-module $\Hmod{E}$ over $\Co{\X}$ and an equivariant C$^*$-algebra isomorphism $\Co{\mathbb{Y}} \rightarrow \mathcal{K}(\Hmod{E})$. We say that such an equivariant Hilbert module $\Hmod{E}$ and associated isomorphism implement the Morita equivalence.\end{Def}

Note that the above terminology is justified by Remark~\ref{RemHilb}.

\begin{Not}\label{DefCatFinEq}
Let $\G$ be a compact quantum group, and $\X$ a quantum homogeneous space over $\G$. We let $\cat{D}_{\X}$ denote the category of finite equivariant Hilbert $\cC^*$-modules over $\X$, whose morphisms are the equivariant adjointable maps between Hilbert $\cC^*$-modules.
\end{Not}

\begin{Prop}
The category $\cat{D}_{\X}$ is a semi-simple $\cC^*$-category.
\end{Prop}

\begin{proof}
By the above proposition, for any object $\Hmod{E}$ in $\cat{D}_\X$, the algebra $\Mor(\Hmod{E}, \Hmod{E})$ is a finite-dimensional $\cC^*$-algebra.  Moreover, if $p\in \Mor(\Hmod{E},\Hmod{E})$ is a projection, then $p\Hmod{E}$ is again an object of $\cat{D}_\X$.  Remark~\ref{RemSSFinDimMor} then implies the assertion.
\end{proof}

In view of Example~\ref{ExaHilb}, it can be seen that finite (resp. irreducible) equivariant Hilbert $\cC^*$-modules play a similar r\^{o}le as the finite-dimensional (resp. irreducible) representations of $\G$.

Now let $\Hmod{E}$ be a finite equivariant Hilbert $\Co{\X}$-module, and let $u$ be a finite-dimensional unitary representation of $\G$. Then we can amplify $\Hmod{E}$ with $u$ to obtain the equivariant Hilbert module $u\Circt \Hmod{E}$.  As a Hilbert $\Co{\X}$-module, $u \Circt \Hmod{E}$ is the amplification $\Hsp_u\otimes \Hmod{E}$ of $\Hmod{E}$ with the Hilbert space $\Hsp_u$. The coaction of $\Co{\G}$ is given by the formula
\[
(u\Circt \alpha_{\Hmod{E}})(\xi \otimes \eta) = u_{1 3} (\xi \otimes \alpha(\eta)),
\]
Then obviously $u \Circt \Hmod{E}$ is still finite.  We record the following facts for later reference.

\begin{Lem}\label{LemEqvrHModStblz}
For any $\Hmod{E} \in \cat{D}_\X$, there exists a representation $u$ of $\G$ for which there is an isometric morphism of $\Hmod{E}$ into $u\Circt \Co{\X}$.
\end{Lem}

\begin{proof}
This is a consequence of the equivariant stabilization, see Section~3.2 of \cite{Ver1}.
\end{proof}

\begin{Prop}
Let $\X$ be a quantum homogeneous space for a compact quantum group $\G$. Denote by $\cat{D}_{\X}$ the $\cC^*$-category of finite equivariant Hilbert $\Co{\X}$-modules. Then the operation
\[
\Rep(\G) \times \cat{D}_\X\rightarrow \cat{D}_{\X}, \quad (u,\Hmod{E}) \mapsto u\Circt \Hmod{E}
\]
defines a connected $\Rep(\G)$-module C$^*$-category structure on $\cat{D}_{\X}$.
\end{Prop}
\begin{proof}
The maps necessary to complete the $\Rep(\G)$-module category structure are obvious, coming from the ordinary associativity maps for the concrete tensor products of the underlying Hilbert spaces and Hilbert $\cC^*$-modules.

Let us prove that $\cat{D}_\X$ is connected over $\Rep(\G)$.  Let $\Hmod{E}$ and $
\Hmod{F}$ be arbitrary objects in $\cat{D}$.  By Lemmas~\ref{LemEqvrHModStblz} and \ref{LemModCatFrob}, we can find a representation $u$ such that $\Co{\X}$ appears inside $u\Circt \Hmod{E}$. Then, again by Lemma~\ref{LemEqvrHModStblz}, we can a suitable representation $v$ such that $\Mor(v\Circt \Hmod{E},\Hmod{F}) \neq 0$. Hence $\cat{D}$ is connected.
\end{proof}

\begin{Rem}\label{RemEqvKgrp}
The equivariant K-group $K_0^\G(\Co{\X})$ is a free abelian group generated by the irreducible classes of $\cat{D}_\X$.  Note that for compact groups, the above picture was already presented, modulo some of the terminology, in~\cite[section 9]{Was2}. Its extension to the compact quantum group setting was treated in~\cite{Tom1}.
\end{Rem}

We aim to show in the next sections that the module $\cC^*$-category $\cat{D}_\X$, together with the distinguished element corresponding to the standard Hilbert $\cC^*$-module $\Co{\X}$, remembers the quantum homogeneous space $\X$.

\section{An algebraic approach to quantum group actions}

In this section, we will provide a characterization of quantum homogeneous spaces and equivariant Hilbert modules with the analysis drained out of it. This intermediate step will make the Tannaka--Kre\u{\i}n machine of the next section run more smoothly.

The main argument provides an algebraic description of an arbitrary action of a compact quantum group $\G$. It is based on results which appear already in~\cite{Boc1,Pod1}.

We first recall the notion of Hopf $^*$-algebra associated with a compact quantum group.

\begin{Def}\cite{Wor2} Let $\G$ be a compact quantum group. If $u$ is a finite-dimensional unitary representation of $\G$, the elements $(\id \otimes \omega_{\xi, \eta})(u) \in \Co{\G}$ for $\xi, \eta \in \mathscr{H}_u$ are called the \emph{matrix coefficients} of $u$.  The set of all such elements with the $u$ ranging over the representations of $\G$ form a dense Hopf $^*$-subalgebra $\PW(\G)\subseteq \Co{\G}$.
\end{Def}

\begin{Def}
Let $\G$ be a compact quantum group. Let $\aA$ be a unital $^*$-algebra. An \emph{algebraic action} of $\G$ on $\aA$ is defined to be a Hopf $^*$-algebra coaction
\[
\alpha_{\aA}\colon \aA\rightarrow \aA\otimes \PW(\G),
\]
the tensor product on the right being the algebraic one, such that $\aA^{\G}$ is a unital $\cC^*$-algebra, and such that the following positivity condition is satisfied:
\begin{quote}
The map $x \mapsto E_{\G}(x) = (\id\otimes \varphi_{\G})\alpha(x)\in \aA^{\G}$ is completely positive on $\aA$. \hfill \textbf{\ConPos}
\end{quote}
\end{Def}

To be clear, the complete positivity means that for any $n\in \mathbb{N}$ and any element $a \in \aA\otimes M_n(\mathbb{C})$, the element $(E_{\G}\otimes \id)(a^*a)$ is a positive element in the $\cC^*$-algebra $\aA^{\G}\otimes M_n(\mathbb{C})$.

\begin{Lem}\label{LemAlgAct}
Let $\G$ be a compact quantum group with an action $\alpha_A$ on a unital $\cC^*$-algebra $A$. Let $\aA$ denote the linear span of $(\id\otimes \varphi_{\mathbb{G}})(\alpha_A(x)(1\otimes g))$ for $x \in A$ and $g\in \PW(\G)$.  Then $\aA$ is a dense unital $^*$-subalgebra of $A$ on which $\alpha_A$ restricts to an algebraic action.
\end{Lem}

\begin{proof} See~\cite[Theorem 1.5]{Pod1}, and~\cite[Lemma~11 and Proposition~14]{Boc1}, whose proofs do not depend on the ergodicity assumption made there. The complete positivity of $E_{\G}$ follows from the way it is defined in \ConPos; namely, $^*$-homomorphisms, states, their amplifications, and their compositions are completely positive.
\end{proof}

\begin{Prop}\label{LemComp}
Let $\G$ be a compact quantum group with an algebraic action $\alpha_{\aA}$ on a unital $^*$-algebra $\aA$. Then there exists a unique $\cC^*$-completion $A$ of $\aA$ to which $\alpha_{\aA}$ extends as a coaction of $\Co{\G}$. Moreover, $A^{\G} = \aA^{\G}$.
\end{Prop}

\begin{proof}
We denote by $B$ the $\cC^*$-algebra $\aA^{\G}$.  By the complete positivity assumption on $E_{\G}$, the $B$-valued inner product $\langle a,b\rangle_B = E_{\G}(a^*b)$ on $\aA$ gives a pre-Hilbert $B$-module structure.  We want to show that the left representation of $\aA$ on itself by left multiplication extends to the Hilbert module completion $\mathcal{A}$.

Let $a$ be an arbitrary element of $\aA$. Since the image of $\alpha_{\aA}$ ends up in the algebraic tensor product of $\aA$ and $\PW(\G)$, there is a finite-dimensional unitary representation $u$ of $\G$ and an intertwiner from $\bar{u}$ to $\aA$ whose image contains $a$.

Let us choose an orthonormal basis $e_i$ of $\mathscr{H}_u$, and put $u_{i j} = (\omega_{e_i,e_j}\otimes\id)(u)$.  Then, the above statement means that there are elements $a_i \in \aA$ such that
\begin{itemize}
\item
$a$ can be written as a linear combination $\sum_i \lambda_i a_i$, and
\item
the elements $a_i$ transform according to $(u_{j i}^*)$, so $\alpha_{\aA}(a_i)=\sum_j a_j \otimes u_{j i}^*$.
\end{itemize}
The unitarity of $u$ implies that $\sum_i a_i^*a_i \in B$.

Since $B$ is a $\cC^*$-algebra, one has the inequality $\sum_i a_i^*a_i \leq \| \sum_i a_i^*a_i\|_B$.  Fix now some $j$. Combining the inequalities $a_j^*a_j \leq \sum_i a_i^*a_i$ in $\aA$ with the previous one, the positivity of $E_{\G}$ implies that
\[
E_{\G}\left(b^* a_j^* a_j b\right) \leq  \Bigl\| \sum_i a_i^*a_i \Bigr\|_B E_{\G}(b^*b),\qquad \forall b\in \aA.
\]
It follows that left multiplication with each $a_j$ is bounded, so that $a$ extends as a left multiplication operator to $\mathcal{A}$.

We obtain in this way a faithful $^*$-representation $\aA\rightarrow \mathcal{L}_B(\mathcal{A})$. Define $A$ to be the norm-completion of $\aA$ in this representation. We claim that the coaction $\alpha_{\aA}$ extends to $A$.  Consider the transformation $X$ on $\aA\otimes \PW(\G)$ defined by $X(a\otimes g) = \alpha_{\aA}(a)(1\otimes g)$.  Then, the invariance of $\varphi_\G$ implies that $X$ extends to a unitary morphism on the right Hilbert $B$-module $\mathcal{A} \otimes \mathscr{L}^2(\G)$.  By a routine computation we obtain that $a \mapsto X(a\otimes 1)X^*$ for $a\in A$ gives the extension $\alpha_A$ of $\alpha_{\aA}$ to $A$.

From this formula for $\alpha_{A}$, it also follows that we have $(\id\otimes \varphi_{\G})\alpha(a) = \langle a\cdot 1_B,1_B\rangle_B$ for all $a\in A$. It follows that the invariant elements of $A$ lie in $B$.

It remains to prove the uniqueness of $A$. Let us assume that $A$ is an arbitrary unital $\cC^*$-algebra satisfying the conclusion of the lemma. Then $E_{\G}$ can, by the same formula, be extended to a conditional expectation from $A$ to $B$. Since $\G$ is reduced, this conditional expectation is faithful.

Now, if $a\in \aA\subseteq A$ and $R < \norm{a}$, the functional calculus shows that there is a positive element $b \in A$ such that $(R b)^2 < b a^* a b$.  Thus, the norm of $a$ can be characterized by
\[
\norm{a} = \sup_{b \in \aA \setminus \{0\}} \left( \frac{\norm{E_{\G}(b^*a^*ab)}}{\norm{E_{\G}(b^*b)}}\right)^{\frac{1}{2}}.
\]
Hence the $\cC^*$-norm on $\aA$ is uniquely determined in terms of $(\aA, \alpha_{\aA})$.
\end{proof}

\begin{Prop}\label{PropComp2}
Let $\G$ be a compact quantum group. Then the correspondences $A\mapsto \aA$ and $\aA \mapsto A$ of Lemma ~\ref{LemAlgAct} and Proposition~\ref{LemComp} can be extended to respective functors $\alg$ and $\comp$ between the categories of actions of $\mathbb{G}$ and algebraic actions of $\G$. Moreover, $\comp\circ \alg$ is naturally equivalent to the identity functor.
\end{Prop}

Here, the morphisms on the respective categories are understood to be the equivariant unital $^*$-homomorphisms.

\begin{proof}
Let $A$ and $B$ be unital $\cC^*$-algebras endowed with $\mathbb{G}$-actions, and let $f\colon A\rightarrow B$ be an equivariant unital $^*$-homomorphism.  The equivariance implies that $f$ restricts to an equivariant $^*$-homomorphism $\aA\rightarrow \mathscr{B}$.  This gives the functor $\alg$.

Conversely, suppose that $\aA$ and $\mathscr{B}$ are unital $^*$-algebras with algebraic $\G$-actions, $A$ and $B$ their respective completions.  Then the direct sum $A \oplus B$ admits a canonical $\G$-action extending the ones on the direct summands.  If $f\colon \aA\rightarrow \mathscr{B}$ is an equivariant unital $^*$-homomorphism, the map $(\id \times f)(a) = a \oplus f(a)$ is a faithful $\G$-equivariant homomorphism from $ \aA$ to $A \oplus B$.  Proposition~\ref{LemComp} implies that the $\cC^*$-norm on $\aA$ induced by $\id \times f$ has to agree with the $A$-norm.  Hence $f$ extends to an equivariant $^*$-homomorphism $A \rightarrow B$.  This way we obtain the functor $\comp$.

Now, the natural equivalence between $\comp\circ \alg$ and the identity functor follows directly from density part in Lemma \ref{LemAlgAct} and the uniqueness part in Proposition~\ref{LemComp}.
\end{proof}

\begin{Rem}\label{RemAlgComp}
The composition $\alg\circ \comp$ is not equivalent to the identity functor in general. For example, if $A$ is given by the function algebra of closed disk $\Co{\bar{\mathbb{D}}}$ endowed with the rotation action of $U(1)$, the algebra $A$ contains many $U(1)$-invariant norm dense subalgebras corresponding to the various decaying conditions around the origin. However, on the subcategory of the actions with finite-dimensional fixed point algebras, $\alg \circ \comp$ is indeed equivalent to the identity functor.
\end{Rem}

\section{Tannaka--Kre\u{\i}n construction}\label{SecTan}

Let $\mathbb{G}$ be a compact quantum group.  We take a set $I$ indexing the equivalence classes of irreducible objects in $\Rep(\G)$, and a distinguished irreducible object $u_a$ for each $a\in I$. When convenient, we will abbreviate $u_a$ by $a$. The index corresponding to the unit object of $\Rep(\G)$ will be written as $o$.  We identify $\mathscr{H}_o$ with $\mathbb{C}$ (canonically) by means of the tensor structure.

It will be handy to use the following Penrose--Einstein-like notation. It concerns the natural map
\begin{equation}\label{EqHuDec}
\bigoplus_{a\in I} \Mor(u_a,u) \otimes \Hsp_a \rightarrow \Hsp_u, \quad \sum_i x_i \otimes \xi_i \mapsto \sum_i x_i(\xi_i)
\end{equation}
for any representation $u$. This map is an isomorphism, see Lemma~\ref{LemEqui}.

\begin{Not}\label{NotPen1} We will write the inverse of~\eqref{EqHuDec} as
\(
\xi \mapsto \xi^a \otimes \xi_a,
\)
so that $\xi = \xi^a \xi_a$.
\end{Not}

For the rest of this section, we will fix a semi-simple $\Rep(\G)$-module $\cC^*$-category $\mathcal{D}$.

\begin{Not} For objects $x, y$ in $\mathcal{D}$, we denote by $\aA_x^y$ the vector space
\[
\aA^y_x = \bigoplus_{a \in I} \Mor(u_a\otimes y,x) \otimes \mathscr{H}_a
\]
\end{Not}

The direct sum on the right hand side is the algebraic one.  We can endow $\aA^y_x$ with the $\PW(\G)$-comodule structure $\alpha_x^y = \oplus_a (\id \otimes \delta_a)$, where $\delta_a$ is defined in Example~\ref{ExaHilb}.

\begin{Rem} The space $\oplus_a \Mor(u_a\otimes y,x) \otimes \mathscr{H}_a$ may be seen as the coend of the functor $\cat{C}^{\mathrm{op}} \times \cat{C} \rightarrow \mathrm{Vect}$ sending $(u,v)$ to $\Mor(u\otimes y,x) \otimes \mathscr{H}_v$, see for instance~\cite[Section 2]{Mug2},~\cite[Chapter~IX]{Mac1}.
\end{Rem}

Our goal is to make the $\aA_y^y$ into algebraic actions for $\G$, and the $\aA^y_x$ into equivariant right pre-Hilbert modules for $\aA_y^y$.

\begin{Not} When $f$ stands for an element in $\aA_x^y$, its leg in $\Mor(u_a\otimes y,x)$ (resp. in $\mathscr{H}_a$) for $a \in I$ is denoted by $f^a$ (resp. $f_a$).  Thus, the expression of the form $f^a \otimes f_a$ is understood to represent $f$.
\end{Not}

We will combine this notation with Notation \ref{NotPen1}. This notation can be seen as analogous to the Sweedler notation for coproducts. As an example, consider fixed $a,b\in I$, and elementary tensors $f = x\otimes \xi$ and $g = y\otimes \eta$ respectively in $\Mor(u_a\otimes y,x)\otimes \Hsp_a$ and $\Mor(u_b\otimes y,x)\otimes \Hsp_b$. Choose a maximal family of mutually orthogonal isometric morphisms $(\iota^c_{a b, k})_k$ from $u_c$ to $u_a \Circt u_b$. Then we have
\[
f^c \otimes g^d \otimes (f_c\otimes g_d)^e \otimes (f_c\otimes g_d)_e = \sum_{c,k} x\otimes y \otimes \iota^c_{a b,k}\otimes (\iota^c_{a b,k})^*(\xi\otimes \eta)
\]
inside $\bigoplus_c \Mor(u_a\otimes y,x)\otimes \Mor(u_b\otimes y,x)\otimes \Mor(u_c,u_a\otimes u_b)\otimes \Hsp_c$.

As an exercise to get acquainted with the notation, the reader could try to prove the following interchange law \[\lbrack(\xi^a\otimes \id_v)(\xi_a\otimes \eta)^c\rbrack \otimes (\xi_a\otimes \eta)_c = (\xi\otimes \eta)^c\otimes (\xi\otimes \eta)_c \cong \xi\otimes \eta,\] where $\xi,\eta$ are arbitrary vectors respectively in ~ $\Hsp_u$ and $\Hsp_v$.

\begin{Def}\label{DefPenEinNotComp} Let $x,y,z$ be objects in $\mathcal{D}$. We define a \emph{multiplication map} \begin{equation}\label{EqMulAxyAyzAxz}\aA_x^y\times \aA_y^z \rightarrow \aA_x^z\end{equation} by the formula
\[ f g = (f g)^c\otimes (f g)_c =  \lbrack f^a(\id_a\otimes g^b)\phi_{a, b, z}((f_a\otimes g_b)^c\otimes \id_z)\rbrack \otimes  (f_a\otimes g_b)_c.
\]
where $\phi_{a, b, z}$ is the associator from Definition~\ref{DefModCat}.
\end{Def}

\begin{Prop}
The multiplication~\eqref{EqMulAxyAyzAxz} is associative.
\end{Prop}

\begin{proof}
Let $(f,g,h) \in  \aA^y_x\times \aA^z_y\times \aA^w_z$.  First, the product $(f g)h$ can be expressed as
\[
\lbrack \lbrack f^a (\id_a \otimes g^b)\phi_{a,b,z}((f_a\otimes g_b)^c\otimes \id_z) \rbrack (\id_c \otimes h^d)\phi_{c,d,w}(((f_a\otimes g_b)_c\otimes h_d)^e\otimes \id_w) \rbrack \otimes ((f_a\otimes g_b)_c \otimes h_d)_e.
\]
Taking composition at $c$ and using naturality of $\phi$, the above is equal to
\[
\lbrack f^a(\id_a \otimes g^b)(\id_a\otimes \id_b\otimes h^d)\phi_{a,b,d\otimes w}\phi_{a\smCirct b,d,w} ((f_a\otimes g_b\otimes h_d)^e\otimes \id_w) \rbrack \otimes (f_a \otimes g_b \otimes h_d)_e.
\]
Similarly, the expression $f(g h)$ reduces to \[
\lbrack f^a(\id_a \otimes g^b)(\id_a\otimes \id_b\otimes h^d)(\id_a\otimes \phi_{b,d,w})\phi_{a,b\smCirct d,w} ((f_a\otimes g_b\otimes h_d)^e\otimes \id_w) \rbrack \otimes (f_a \otimes g_b \otimes h_d)_e.
\]
The conclusion then follows from the associativity constraint on $\phi$.
\end{proof}

\begin{Prop}\label{PropUnit}
Let $x$ and $y$ be objects in $\mathcal{D}$, and let $e_y\in \Mor(u_o\otimes y,y)$ be the structure map of tensor unit included in the module package.  Then the element $1_y = e_y \otimes 1 \in \aA_y^y$ is a right unit for the multiplication map $\aA_x^y\times \aA_y^y\rightarrow \aA_x^y$, and a left unit for the multiplication map $\aA_y^y\times \aA_y^x \rightarrow \aA_y^x$.
\end{Prop}

\begin{proof}
Take $f \in \aA_y^x$. Then the formula for the product $f\cdot 1_y$ reads
\begin{equation*}
\lbrack f^a(\id_a\otimes e_y)\phi_{a,o,y}((f_a\otimes 1)^c\otimes \id_y)\rbrack \otimes (f_a \otimes 1)_c.
\end{equation*}
Since $\Mor(u_c, u_o \Circt u_a) \neq 0$ if and only if $a = c$ for $a, c \in I$, the unit constraint on $e$ reduces this expression to $ f^a(\id_a\otimes \id_y)\otimes f_a = f$. This shows that $1_y$ is a left unit.  An  analogous argument shows that $1_y$ is also a left unit.
\end{proof}

It follows that we can make a category $\mathscr{A}$ having the same objects as $\cat{D}$, and with morphism space from $x$ to $y$ the linear space $\aA_y^x$. In particular the `endomorphism spaces' $\aA_y^y$ are unital algebras. It contains $\cat{D}$ as a faithful sub-$^*$category, as shown by the following lemma.

\begin{Lem}\label{LemMorphi} There is a linear functor $\mathcal{D}\rightarrow \cat{A}$ which is the identity on objects, and which sends $f\in \Mor(y,x)$ to $f e_y\otimes 1 \in \mathscr{A}_x^y$.
\end{Lem}

\begin{proof} This is proven in the same way as Proposition \ref{PropUnit}.
\end{proof}

In the following, we will identify $\Mor(y,x)$ with its image inside $\aA_x^y$.

\begin{Prop}\label{PropCoactAlgHomAxy}
Take $x,y,z$ objects in $\mathcal{D}$. Let $f\in \aA_x^y$ and $g\in \aA_y^z$. Then
\[
\alpha_x^z(f g) = \alpha_x^y(f)\alpha_y^z(g).
\]
\end{Prop}

\begin{proof}
When $(E, \alpha)$ a right comodule over $\PW(\G)$, let us write, for $x\in E$, \[\alpha(x) = x_{(0)} \otimes x_{(1)} \in E \otimes \PW(\G).\] Then, resorting again to the notation of Example \ref{ExaHilb}, one has \[\delta_{u \smCirct v}(\xi \otimes \eta) = \xi_{(0)} \otimes \eta_{(0)} \otimes \xi_{(1)} \eta_{(1)}.\] Using that $\xi^c \otimes \delta_c(\xi_c) = (\xi_{(0)})^c \otimes (\xi_{(0)})_c\otimes \xi_{(1)}$, the element $\alpha_x^z(f g)$ can thus be computed as
\begin{multline*}
\lbrack f^a(\id_a \otimes g^b)\phi_{a,b,z}(f_a\otimes g_b)^c \rbrack \otimes \delta_c((f_a \otimes g_b)_c)\\ = \lbrack f^a(\id_a \otimes g^b)\phi_{a,b,z}(f_{a(0)}\otimes g_{b(0)})^c \rbrack \otimes (f_{a(0)} \otimes g_{b(0)})_c \otimes f_{a(1)} g_{b(1)}.
\end{multline*}
On the other hand, the way the coaction $\alpha_x^y$ is defined implies that
\[
f^a \otimes f_{a(0)} \otimes f_{a(1)} = (f_{(0)})^a \otimes (f_{(0)})_a \otimes f_{(1)}.
\]
It follows that $\alpha_x^z(f g)$ can be expressed as
\[
\lbrack (f_{(0)})^a(\id_a \otimes (g_{(0)})^b)\phi_{a,b,z}((f_{(0)})_a\otimes (g_{(0)})_b)^c \rbrack \otimes ((f_{(0)})_a \otimes (g_{(0)})_b)_c \otimes f_{(1)} g_{(1)},
\]
which is precisely $\alpha_x^y(f)\alpha_y^z(g)$.
\end{proof}

We will now define a $^*$-operation $\aA_x^y \rightarrow \aA_y^x$. Here the rigidity of $\Rep(\G)$ will come into play, so we first fix our conventions concerning duals.

\begin{Not}
When $f^u \in \Mor(u\otimes y,x)$, we write $\lu{u}f \in \Mor(y, \bar{u} \otimes x)$ for its image of the Frobenius isomorphism associated with $(R_u, \bar{R}_u)$ (see Lemma~\ref{LemModCatFrob}). So, \[\lu{u}f = (\id_{\bar{u}}\otimes f^u)\phi_{\bar{u},u,y}(R_u\otimes \id_y)e_y^*.\] Similarly, when $ \xi_u \in \Hsp_u$, we define $\ld{u}\xi \in \overline{\Hsp_{\bar{u}}}$ by the formula \[ \overline{\ld{u}\xi} =  (\xi_u^* \otimes \id_{u}) \bar{R}_{u}(1),\] where $\xi^*$ for a vector $\xi \in \Hsp$ is the obvious map $\Hsp\rightarrow \mathbb{C}$.
\end{Not}

\begin{Def}\label{DefStar}
We define the anti-linear \emph{conjugation map} $^*\colon \aA_x^y\rightarrow \aA^x_y$ by
\[
f^* = (f^*)^{\bar{a}} \otimes (f^*)_{\bar{a}} =  (\lu{a}f)^* \otimes \overline{\ld{a}f}.
\]
\end{Def}

Since the above formula involves both $R_a$ and $\bar{R}_a^*$ for each $a \in I$, the definition of $\,^*\,$ is actually independent of the choice of the duality morphisms.

\begin{Prop}
The operation $^*$ is anti-multiplicative.
\end{Prop}

\begin{proof}
Let $f\in \aA_x^y$ and $g\in \aA_y^z$. Then by definition of the product,
\begin{align*}
(g^*f^*)^{\bar{c}}\otimes (g^*f^*)_{\bar{c}} &= \lbrack (g^*)^{\bar{a}}(\id_{\bar{a}}\otimes (f^*)^{\bar{b}})\phi_{\bar{a},\bar{b},x}(((g^*)_{\bar{a}}\otimes (f^*)_{\bar{b}})^{\bar{c}}\otimes \id_x)\rbrack \otimes ((g^*)_{\bar{a}}\otimes (f^*)_{\bar{b}})_{\bar{c}}\\
&= \lbrack (\lu{a}g)^{*}(\id_{\bar{a}}\otimes (\lu{b}f)^*)\phi_{\bar{a},\bar{b},x}((\overline{\ld{a}g}\otimes \overline{\ld{b}f})^{\bar{c}}\otimes \id_x) \rbrack \otimes (\overline{\ld{a}g}\otimes \overline{\ld{b}f})_{\bar{c}}.
\end{align*}
Let us concentrate first on the part $\phi_{\bar{a},\bar{b},x}^*(\id_{\bar{a}}\otimes \lu{b}f)\lu{a}g$. Choose as solution for the conjugate equations for $b\Circt a$ the couple $((\id_{\bar{a}}\otimes R_b\otimes \id_a)R_a, (\id_{a} \otimes \bar{R}_b\otimes \id_{\bar{a}})\bar{R}_a)$. Then, using naturality and coherence for $\phi$ and $e$, we can write, after some diagram manipulations,
\begin{equation*}
\phi_{\bar{a},\bar{b},x}^*(\id_{\bar{a}}\otimes \lu{b}f)\lu{a}g = (\id_{\bar{a}\smCirct \bar{b}}\otimes (f^b(\id_b\otimes g^a)))\phi_{\bar{a}\smCirct\bar{b},b,a\otimes z}\phi_{\bar{a}\smCirct \bar{b}\smCirct b,a,z}(R_{b\smCirct a}\otimes \id_z)e_z^*.
\end{equation*}
Substituting in the expression for $g^*f^*$ and pulling through the factor $((g^*)_a\otimes (f^*)_b)^c\otimes \id_x$, we find that $g^*f^*$ is equal to the expression
\[
\lbrack e_z(R_{b\smCirct a}^*\otimes \id_z)((\overline{\ld{a}g}\otimes \overline{\ld{b}f})^{\bar{c}}\otimes\id_{b\smCirct a}\otimes \id_z)\phi_{\bar{c}\otimes b,a,z}^*\phi_{\bar{c},b,a\otimes z}^*(\id_{\bar{c}}\otimes ((\id_b\otimes g^{a*})f^{b*}))\rbrack \otimes (\overline{\ld{a}g}\otimes \overline{\ld{b}f})_{\bar{c}}.
\]
Now for vectors $\xi$ and $\eta$ in representation spaces, we have
\[
\lbrack R_{b\smCirct a}^*((\overline{\ld{a}\xi}\otimes \overline{\ld{b}\eta})^{\bar{c}}\otimes \id_b\otimes \id_a)\rbrack \otimes (\overline{\ld{a}\xi}\otimes \overline{\ld{b}\eta})_{\bar{c}} = \lbrack R_c^*(\id_{\bar{c}}\otimes ((\eta_b\otimes \xi_a)^c)^*)\rbrack \otimes \overline{\ld{c}(\eta_b\otimes\xi_a)},\] which can be verified using the natural isomorphism \[\oplus_c \Mor(\bar{c}\Circt b\Circt a,\mathbbm{1})\otimes \Hsp_{\bar{c}}\rightarrow \overline{\Hsp}_b\otimes \overline{\Hsp}_a
\]
and the conjugate equations for $(R,\bar{R})$. It follows that $g^*f^*$ can be written as
\[
\lbrack e_z(R_c^*\otimes \id_z)(\id_{\bar{c}}\otimes ((f_b\otimes g_a)^{c})^*\otimes \id_z)\phi_{\bar{c}\otimes b,a,z}^*\phi_{\bar{c},b,a\otimes z}^*(\id_{\bar{c}}\otimes ((\id_b\otimes g^{a*})f^{b*}))\rbrack \otimes \overline{\ld{c}(f_b\otimes g_a)}.
\]
Using once more coherence and naturality for $\phi$, this reduces to $(f g)^*$.
\end{proof}

\begin{Prop}
The operation $^*$ is involutive.
\end{Prop}

\begin{proof}
Let $f \in \aA_x^y$.  By the definition of the $^*$-operation, $(f^*)^*$ can be written as
\begin{multline*}
\lbrack e_x(R_{\bar{a}}^*\otimes \id_x)\phi_{a,\bar{a},x}^*(\id_{a}\otimes \id_{\bar{a}}\otimes f^{a})(\id_a\otimes \phi_{\bar{a},a,y})(\id_a\otimes R_a\otimes \id_y)(\id_a\otimes e_y^*)\rbrack \\ \otimes (\bar{R}_a^*\otimes \id_a)(f_a\otimes \bar{R}_{\bar{a}}(1)).
\end{multline*}

Using again naturality and coherence for $\phi$ and $e$, this can be rewritten
\begin{equation*}
(f^*)^* = \lbrack f^a (R_{\bar{a}}^*\otimes\id_a\otimes \id_y)(\id_a\otimes R_a\otimes \id_y)\rbrack  \otimes (\bar{R}_a^*\otimes \id_a)(f_a\otimes \bar{R}_{\bar{a}}(1)).
\end{equation*}
But since we may replace the conjugate solution $(R_{\bar{a}},\bar{R}_{\bar{a}})$ with $(\bar{R}_a,R_a)$, the conjugate equations for $(R_a,\bar{R}_a)$ show that the above expression reduces to $f$.
\end{proof}

\begin{Prop}\label{LemStarCoact}
For $f\in \aA_x^y$, we have $\alpha_x^y(f)^* = \alpha_y^x(f^*)$.
\end{Prop}

\begin{proof}
The coaction on $f^*$ can be written as
\[
\lbrack e_y(R_a^* \otimes \id_y) \phi^*_{\bar{a},a,y}(\id_{\bar{a}} \otimes f^{a *})\rbrack \otimes (f_a^* \otimes u_{\bar{a}}) (\bar{R}_a(1) \otimes 1).
\]
Since $\bar{R}_a \in \Mor(u_o, u_a \Circt u_{\bar{a}})$, one has
\[
 (u_{\bar{a}})_{2 3} (\bar{R}_a)_{1 2} = (u_a^*)_{1 3} (u_a)_{1 3} (u_{\bar{a}})_{2 3} (\bar{R}_a)_{1 2} = (u_a^*)_{1 3} (\bar{R}_a)_{1 2}.
\]
Thus, we obtain
\[
\alpha_y^x(f^*) = \lbrack e_y(R_a^* \otimes \id_y) \phi^*_{\bar{a},a,y}(\id_{\bar{a}} \otimes f^{a *})\rbrack \otimes (u(f_a\otimes 1))_{13}^* (\bar{R}_a(1) \otimes 1) = \alpha_x^y(f)^*,
\]
which proves the assertion.
\end{proof}

\begin{Lem}\label{LemMat}
There is a natural equivariant $^*$-isomorphism
\[
\aA_{x\oplus y}^{x\oplus y} \cong \begin{pmatrix}\aA_x^x & \aA_x^y \\ \aA_y^x & \aA_y^y\end{pmatrix}.
\]
\end{Lem}

\begin{proof}
This follows from the natural decomposition
\[
\End(x\oplus y)\cong \begin{pmatrix} \End(x)& \Mor(y,x)\\ \Mor(x,y)&\End(y)\end{pmatrix},
\]
which passes through all further structure imposed on the $\aA_x^y$.
\end{proof}

\begin{Lem}\label{LemInv}
We have $(\aA_x^y)^{\G} = \Mor(y, x)$. Furthermore, for $f\in \aA_y^y$, we have \[(\id \otimes \varphi_{\G})(\alpha_y^y(f)) = f^o f_o e_y^* \in \End(y).\]
\end{Lem}

\begin{proof}
These formulas follow from the definition of $\alpha_x^y$ and the orthogonality of irreducible representations.
\end{proof}

\begin{Theorem}\label{ThmAlgActionFromModCat}
For each object $y$ of $\mathcal{D}$, the coaction of $\PW(\G)$ on $\aA_y^y$ defines an algebraic action of $\G$.
\end{Theorem}

\begin{proof}
The only thing left to prove is the complete positivity \ConPos\@ for the map $E_{\G} = (\id \otimes \varphi_{\G})\circ \alpha_y^y$.  By Lemma~\ref{LemMat}, it is enough to show that $E_\G$ is positive on $\aA_y^y$ for arbitrary $y$.  Let $f, g\in \aA_y^y$.  Then we have
\begin{multline*}
f^* g = \lbrack e_y (R_a^* \otimes \id_y)\phi_{\bar{a},a,y}^* (\id_{\bar{a}} \otimes f^{a *}g^b)\phi_{\bar{a},a,y}((((f_a^* \otimes \id_{\bar{a}}) \bar{R}_a(1) \otimes g_b))^c\otimes\id_y) \rbrack \\ \otimes ((f_a^* \otimes \id_{\bar{a}}) \bar{R}_a(1) \otimes g_b)_c.
\end{multline*}
Applying $E_\G$ to this means taking the value at $c = o$.

Since $u_a$ and $u_b$ are irreducible, there exists an embedding of $u_o$ into $u_{\bar{a}} \Circt u_b$ if and only if $b = a$. In that case an isometric embedding is given by $(\dim_q u_a)^{-1/2} R_a$ for the normalized choice of $(R_a, \bar{R}_a)$.  Thus, we obtain, using the conjugate equations for $(R_a,\bar{R}_a)$ in the last step,
\begin{align*}
((f_a^* \otimes \id_{\bar{a}}) \bar{R}_a(1) \otimes g_a)^o((f_a^* \otimes \id_{\bar{a}}) \bar{R}_a(1) \otimes g_a)_o &=  \frac{1}{\dim_q u_a} (f_a^* \otimes R_a^*) (\bar{R}_a(1) \otimes g_a) R_a \\
&= \frac{\langle f_a,g_a\rangle}{\dim_q u_a}R_a
\end{align*}
as a morphism from $u_o$ to $u_{\bar{a}}\Circt u_a$. Hence,
\begin{align*}
E_\G(f^* g) &= \frac{ \langle f_a,g_a\rangle}{\dim_q u_a} \cdot e_y(R_a^* \otimes \id_y) \phi_{\bar{a},a,y}^*(\id_{\bar{a}} \otimes f^{a *} g^a)\phi_{\bar{a},a,y} (R_a \otimes \id_y)e_y^*\\
&= \frac{1}{\dim_q u_a}  \cdot e_y(R_a^* \otimes \id_y) \phi_{\bar{a},a,y}^*(\id_{\bar{a}} \otimes \langle f,g\rangle_{\Mor(y,y)}) \phi_{\bar{a},a,y} (R_a \otimes \id_y)e_y^*,
\end{align*}
where $\langle f,g\rangle_{\Mor(y,y)} = \langle f_a,g_a\rangle f^{a*}g^a$ is the standard $\Mor(y,y)$-valued inner product on $\aA_x^y$. From this formula, it follows that $E_{\G}$ is indeed completely positive.
\end{proof}

\begin{Rem} In \cite{Ost1}, the construction of an action from a module category is carried out internally within the tensor category. There are two obstacles for attempting such a construction in our setting. The first obstacle is a finiteness problem, in that the algebra underlying an ergodic action will in general live inside a completion of the tensor category. This could be taken care of by standard techniques. The second obstacle is that we want our algebras to be endowed with a good $^*$-structure. Now ergodic actions on \emph{finite-dimensional} C$^*$-algebras can be characterized abstractly inside of $\Rep(\G)$ as (irreducible) abstract $Q$-systems (\cite{Lon1}, \cite{Mug3}). However, the definition of $Q$-system is too restrictive if we want to allow non-finite quantum homogeneous spaces. So although it seems manageable to lift both of the above obstacles separately, we do not know how to tackle them in combination.
\end{Rem}

At this stage, we can apply the material developed in the previous section.

\begin{Not}\label{NotBlock}
For each object $y$ in $\mathcal{D}$, we denote the $\G$-$\cC^*$-algebraic completion of $\aA_y^y$ (see Proposition~\ref{LemComp}) by $A_y^y$. We denote the block decomposition of $A_{x \oplus y}$ induced by the isomorphism of Lemma~\ref{LemMat} as
\[
A_{x\oplus y}^{x\oplus y} = \begin{pmatrix} A_x^x & A_x^y \\ A_y^x & A_y^y\end{pmatrix}.
\]
\end{Not}

In this way, for general $x,y$, the space $A_y^x$ naturally has the structure of an equivariant right Hilbert $A_y^y$-module, together with a unital $^*$-homomorphism from $A_x^x$ into $\mathcal{L}_{A_y^y}(A_x^y)$.

\begin{Lem}\label{PropFinHilb} When $x$ and $y$ are objects in $\cat{D}$ with $y$ irreducible, then the action of $\G$ on $A_y^y$ is ergodic, and $A_x^y$ is a finite equivariant Hilbert $A_y^y$-module.
\end{Lem}

\begin{proof} From the block decomposition as in Notation \ref{NotBlock}, we may as well suppose that also $x$ is irreducible. Then by Lemma \ref{LemInv} and Proposition \ref{LemComp}, we obtain that the actions on $A_y^y$ and $A_x^x$ are ergodic. Since the image of $A_x^x$ in $\mathcal{L}(A_x^y)$ must by construction contain $\mathcal{K}_{A_y^y}(A_y^y)$, we deduce from Remark \ref{RemHilb} that either we have an identification $A_x^x \cong \mathcal{K}_{A_y^y}(A_x^y)$, in which case $A_x^y$ is in particular finitely generated projective, or else $A_x^y = 0$.
\end{proof}

The $A_x^y$ are Banach spaces with the $^*$-operations $A_x^y \rightarrow A_y^x$ satisfying the $\cC^*$-condition.  It follows that we can make a $\cC^*$-category $A$ having the same objects as $\cat{D}$, and with morphism space from $x$ to $y$ given by the Banach space $A_y^x$. By Lemma \ref{LemMorphi}, it contains a faithful copy of the $\cC^*$-category $\mathcal{D}$, which are precisely the fixed points under the $\G$-action on the morphism spaces.

\begin{Prop}\label{LemDcatIntoAcat} Let $y$ be a fixed irreducible object in $\cat{D}$, and let $A^y$ be the category with \begin{itemize} \item[$\bullet$] objects the $A_x^y$, where $x$ ranges over the objects in $\cat{D}$, and
\item[$\bullet$] with morphism space $\Mor_{A^y}(z,x)$ the space $\mathcal{K}_{A_y^y}(A_z^y,A_x^y)$.
\end{itemize}
Then we have a $\cC^*$-functor $F_y\colon A \rightarrow A^y$, sending $x$ to $A_x^y$ and an element $f\in A_x^z$ to left multiplication with this element. Moreover, the resulting maps $\Mor_A(z,x)\rightarrow \mathcal{K}_{A_y^y}(A_z^y,A_x^y)$ are $\G$-equivariant.
\end{Prop}

\begin{proof}
Since the modules $A_x^y$ are finitely generated projective over $A_y^y$, left multiplication with elements in $A_x^z$ indeed gives compact operators from $A_z^y$ to $A_x^y$. The functoriality of the given map is then a formality to check. The equivariance follows from Proposition \ref{PropCoactAlgHomAxy}.
\end{proof}

\begin{Exa}\label{ExaQHomQSubGrp}
Let $\mathbb{H}$ be a quantum subgroup of $\G$.  We have seen in Example~\ref{ExaQSubGrpEtc} that $\Rep(\mathbb{H})$ is a $\Rep(\G)$-module category.  When $w$ is an irreducible unitary representation of $\mathbb{H}$, we find that
\begin{equation*}
\mathscr{A}_w^w \cong \Mor((u_a)_{\mid \mathbb{H}}\otimes w,w)\otimes \Hsp_a \cong (\bar{w}\otimes (\bar{u}_a)_{\mid \mathbb{H}}\otimes w)^{\mathbb{H}}\otimes \Hsp_a \cong (B(\Hsp_w)\otimes P(\G))^{\mathbb{H}},
\end{equation*}
the fixed points being with respect to the $w$-induced left $\mathbb{H}$-action on $B(\Hsp_w)\otimes P(\G)$. It then follows that the action of $\G$ on $\Co{\X_w}$ given by Lemma \ref{PropFinHilb} is equal to the right translation action on the fixed point algebra $(B(\Hsp_w) \otimes \Co{\G})^\mathbb{H}$.
\end{Exa}

\section{Correspondence between the constructions}
\label{SecTanKr}

Let $\G$ be a compact quantum group, and let $\X$ be a quantum homogeneous space over $\G$.  It is known~\cite{Pin1} that the $\G$-algebra $\Co{\X}$ can be recovered from the associated `spectral functor'
\[
u \mapsto \Hom_\G(\Hsp_u, \Co{\X})
\]
on $\Rep(\G)$, where the right hand side simply means the space of $\G$-equivariant linear maps.  In general, if we ignore the problem of completion, any right comodule $\Hmod{E}$ over $\Co{\G}$ can be recovered from its spectral functor by the formula
\begin{equation}\label{EqComdAsSpecFuncTimesForgetful}
\bigoplus_{a \in I} \Hom_{\G}(\mathscr{H}_{a}, \Hmod{E}) \otimes \Hsp_a \simeq \Hmod{E},
\end{equation}
up to completion. The algebra structure of $\Co{\X}$ was recovered from the usual tensor structure on the forgetful functor of $\Rep(\G)$, and the `quasi-tensor' structure on the spectral functor.

The above general scheme and our construction of $\G$-algebra in the previous section are related by the following simple translation.

\begin{Lem}\label{LemSpecFuncModMor}
Let $u \in \Rep(\G)$, and let $(\Hmod{E}, \alpha_{\Hmod{E}})$ be a $\G$-equivariant Hilbert $\cC^*$-module over $\Co{\X}$.  Then one has a natural isomorphism
\begin{equation}\label{EqScExtResAdj}
 \Hom_\G(\Hsp_u, \Hmod{E}) \simeq \Hom_{\G, \Co{\X}}(\Hsp_u \otimes \Co{\X}, \Hmod{E}),
\end{equation}
where the right hand side denotes the space of linear $\G$-equivariant, right $\Co{\X}$-linear maps.
\end{Lem}

\begin{proof}
If $T \in \Hom_\G(\Hsp_u, \Hmod{E})$, the map $\xi \otimes x \mapsto T(\xi)x$ from $\Hsp_u \otimes \Co{\X}$ to $\Hmod{E}$ is $\G$-equivariant and right $\Co{\X}$-linear.  On the other hand, the inverse correspondence is given by pulling back with the embedding $\Hsp_u \rightarrow \Hsp_u \otimes \Co{\X}, \xi \mapsto \xi \otimes 1$.
\end{proof}

The above isomorphism can be regarded as an adjunction between the `scalar extension by $\Co{\X}$' functor and the `scalar restriction' functor (forgetting the action of $\Co{\X}$).  Moreover, $\Co{\X}$ itself can be regarded as an irreducible object in the category $\cat{D}_\X$ by the ergodicity.  Hence, if $\Hmod{E}$ is a finite equivariant Hilbert module over $\Co{\X}$, we have for the right hand side of~\eqref{EqScExtResAdj} that \[ \Hom_{\G, \Co{\X}}(\Hsp_u \otimes \Co{\X}, \Hmod{E}) = \Mor(u\otimes\Co{\X},\Hmod{E}),\] the latter a morphism space in $\cat{D}_\X$. We use here implicitly that adjointability is automatic for $\Co{\X}$-module maps between finitely generated projective modules).

In the following, we use Notation \ref{NotBlock}.

\begin{Prop}\label{PropCatModToAlg}
Let $\bullet$ denote the object $\Co{\X}$ in $\cat{D}_{\X}$. Then the $\G$-$\cC^*$-algebra $A_{\bullet}^{\bullet}$ is equivariantly isomorphic to $\Co{\X}$.  This isomorphism is induced by the embedding
\begin{equation}\label{EqAdotdotCXisom}
\aA_{\bullet}^{\bullet} \rightarrow \Co{\X},\quad f\mapsto f^a(f_a\otimes 1).
\end{equation}
\end{Prop}

\begin{proof}
By Lemma~\ref{LemSpecFuncModMor}, $\aA_\bullet^\bullet$ can be identified with $\oplus_a \Hom_\G(\mathscr{H}_a, \Co{\X}) \otimes \Hsp_a$, and the map~\eqref{EqAdotdotCXisom} is identified with the canonical embedding~\eqref{EqComdAsSpecFuncTimesForgetful}.  We obtain the assertion by comparing our product structure on $\aA_\bullet^\bullet$ with the one in~\cite[Theorem~8.1]{Pin1}.
\end{proof}

\begin{Prop}\label{PropModCatToCatMod}
Let $\mathcal{D}$ be a connected module $\cC^*$-category over $\Rep(\G)$. Let $y\in \mathcal{D}$ be an irreducible object, and write $A_y^y = \Co{\X_y}$. Then there is an equivalence of $\Rep(\G)$-module $\cC^*$-categories $\mathcal{D} \cong \mathcal{D}_{\X_y}$, by restricting the functor $F_y$ from Proposition \ref{LemDcatIntoAcat} to $\cat{D}$.
\end{Prop}

\begin{proof}
First of all, Lemma \ref{PropFinHilb} ensures us that $F_y$ has the proper range on objects. Since $\cat{D}$ is realized inside the category $A$ by taking the $\G$-invariants in morphism spaces, the equivariance part of Proposition \ref{LemDcatIntoAcat} ensures that $F_y$ also has the proper range on morphisms. In the following, we will mean by $F_y$ its restriction to $\cat{D}$.

We next show that $F_y$ is a $\Rep(\G)$-module homomorphism. Let $u$ be a finite-dimensional representation of $\G$, and let $x$ be an object in $\cat{D}$.  Then, the spectral subspace functors associated with $A_{u \otimes x}^y$ and $u \otimes A_{x}^y$ are the same: the one for $A_{u \otimes x}^y$ is, by definition, determined by the spaces $(\Mor(u_a \otimes y, u \otimes x))_{a \in I}$, but the Frobenius isomorphism implies that these are equal to
\[
\Mor((\bar{u} \Circt u_a) \otimes y, x) \simeq \Mor(\bar{u} \Circt u_a, A_x^y) = \Mor(u_a, \Hsp_u \otimes A_x^y)
\]
for $a \in I$.  The resulting linear isomorphism $\aA_{u \otimes x}^y \rightarrow \Hsp_u \otimes \aA_x^y$ is by construction a $\G$-homomorphism. It is right $\mathscr{A}_y^y$-linear and isometric by the same type of calculation as in the previous section. The coherence conditions for $F_y$ follow from the naturality for scalar restriction/extension and from the fact that we can canonically take $\overline{u\Circt v} = \bar{v}\Circt \bar{u}$ using the chosen duality morphisms for $u$ and $v$.

It remains to show that the sets of irreducible classes are in bijection under the functor $F_y$.  By the connectedness of $\cat{D}$, for any object $x$, there exists an (irreducible) representation $u$ of $\G$ such that $\Mor(u\otimes y, x)\neq 0$. Hence $A_x^y$ is a non-zero Hilbert module. As in the proof of Lemma \ref{PropFinHilb}, it follows that $A_x^y$ is irreducible if $x$ is irreducible. If further $x$ and $z$ are irreducible, we must have by the same reasoning that the map \[\begin{pmatrix} A_x^x & A_x^z \\ A_z^x & A_z^z\end{pmatrix} \rightarrow \begin{pmatrix} \mathcal{K}(A^y_x) & \mathcal{K}(A^y_z,A^y_x)\\ \mathcal{K}(A^y_x,A^y_z) & \mathcal{K}(A^y_z)\end{pmatrix}\] is an isomorphism. Using Lemma~\ref{LemInv}, we see that if $x$ and $z$ are non-isomorphic irreducible objects, $A_x^y$ and $A_z^y$ are not equivalent in $\cat{D}_{\X_y}$.

Now, any object in $\cat{D}_{\X_y}$ is a subobject of $u\Circt \Co{\X_y}$ for some finite-dimensional representation $u$ of $\G$. As $F_y$ preserves the module structure, and as $\Co{\X_y}$ is the image of $y$ by construction, we find that any object of $\cat{D}_{\X_y}$ is isomorphic to an object in the image of $F_y$. By Lemma~\ref{LemEquiii}, we conclude that $F_y$ is an equivalence of $\Rep(\G)$-module $\cC^*$-categories.
\end{proof}

To conclude this section, we summarize our main result in the following theorem, which will also include the formalism on bi-graded Hilbert spaces developed in the Appendix.  Indeed, in our setup, abstract module $\cC^*$-categories will arise naturally from the study of quantum homogeneous spaces, and one then passes to the bi-graded Hilbert space picture to reveal the combinatorial structure in a more tangible form, cf. the remark after Theorem 1.5 in~\cite{Gro1}. This will be exploited in our forthcoming paper~\cite{DeC3} to classify the ergodic actions of the quantum $\SU_q(2)$ groups for $0<|q|\leq 1$.

\begin{Theorem}\label{TheoQHmgenModCatStTensorEndo}
Let $\G$ be a compact quantum group. There is a one-to-one correspondence between the following notions.
\begin{enumerate} \item Ergodic actions of $\G$ (modulo equivariant Morita equivalence).
\item Connected module $\cC^*$-categories over $\Rep(\G)$ (modulo module equivalence).
\item Connected strong tensor functors from $\Rep(\G)$ into bi-graded Hilbert spaces (modulo natural tensor equivalence).
\end{enumerate}
\end{Theorem}

The connectedness of a strong tensor functor $F$ into $J$-bi-graded Hilbert spaces means that it can can not be decomposed as a direct sum $F_1\oplus F_2$ with the $F_i$ strong tensor functors into $J_i$-bi-graded Hilbert spaces, $J = J_1\cup J_2$ with $J_1$ and $J_2$ disjoint.

\begin{proof}
The equivalence between the first two structures is a direct consequence of Propositions~\ref{PropCatModToAlg} and~\ref{PropModCatToCatMod}, where the arbitrariness of the choice of irreducible object corresponds precisely to equivariant Morita equivalence, cf. the remark above Notation \ref{DefCatFinEq}. The equivalence between the last two is a consequence of Proposition~\ref{PropCorrModTen}, under which the connectedness can be easily seen to be preserved.
\end{proof}

Let us give a little more detail on the direct correspondence between tensor functors and ergodic actions. Let $J$ be a set, and $(F_{r s})_{r, s \in J}$ be a connected strong tensor functor from $\Rep(\G)$ into column-finite $J$-bi-graded Hilbert spaces.  Then by Proposition \ref{PropCorrModTen}, $\Hf^J$ has a structure of $\Rep(\G)$-module $\cC^*$-category, in such a way that $F_{r s}(u) \cong \Mor(x_r,u\otimes x_s)$. Hence for $r,s$ elements of $J$, the spaces $\aA_{x_r}^{x_s}$ which were constructed in Section~\ref{SecTan} can be explicitly expressed as
\[
\aA_{x_r}^{x_s} = \bigoplus_{a\in I} \overline{F_{rs}(a)} \otimes \Hsp_a,
\]
since we can identify $\Mor(u\otimes x_s,x_r)$ with the conjugate Hilbert space of $\Mor(x_r,u\otimes x_s)$ by means of the adjoint map.

\section{Categorical description of equivariant maps}
\label{SecEquivarMap}

In this last section, we investigate the relationship between equivariant maps between quantum homogeneous spaces and equivariant functors between module C$^*$-categories.

Let $\X$ and $\Y$ be quantum homogeneous spaces over $\G$, respectively given by the coactions $\alpha \colon \Co{\X} \rightarrow \Co{\X} \otimes \Co{\G}$ and $\beta \colon \Co{\Y} \rightarrow \Co{\Y} \otimes \Co{\G}$.  A $\G$-morphism from $\Y$ to $\X$ is represented by a unital $^*$-algebra homomorphism $\theta$ from $\Co{\X}$ to $\Co{\Y}$ satisfying the $\G$-equivariance condition $(\theta \otimes \id) \circ \alpha = \beta \circ \theta$.

Given such a homomorphism $\theta$, we obtain a $^*$-preserving functor $\theta_\#\colon \cat{D}_\X \rightarrow \cat{D}_\Y$ defined as the extension of scalars $\Hmod{E} \mapsto \Hmod{E} \otimes_{\Co{\X}} \ld{\theta}{\Co{\Y}}$.  We may assume that this functor maps the distinguished object $\Co{\X}$ of $\cat{D}_\X$ to the one of $\cat{D}_\Y$, namely $\Co{\Y}$.  When $u \in \Rep(\G)$ and $\Hmod{E} \in \cat{D}_\X$, let $\psi_{\theta}$ denote the isomorphism
\[
 (\Hsp_u \otimes \Hmod{E}) \otimes_{\Co{\X}} \Co{\Y}\rightarrow \Hsp_u \otimes (\Hmod{E} \otimes_{\Co{\X}} \Co{\Y}), \quad (\xi \otimes x) \otimes y \mapsto \xi \otimes (x \otimes y).
\]
Then $\psi_{\theta}$ can be considered as a natural unitary transformation $\psi_{\theta}\colon \theta_\# ( -\otimes -) \rightarrow - \otimes (\theta_\# -)$ between functors from $\Rep(\G) \times \cat{D}_\X$ to $\cat{D}_\Y$.  This $\psi_{\theta}$ enables one to complete $\theta_\#$ to a module $\cC^*$-category homomorphism between $\cat{D}_\X$ and $\cat{D}_\Y$, cf. Definition \ref{DefMorMod}.

We aim to characterize the $\G$-equivariant morphisms of quantum homogeneous spaces in terms of their associated categories and functors between them.

\begin{Theorem}\label{ThmEquivHomFctor}
Let $\X$ and $\Y$ be quantum homogeneous spaces over $\G$.  Let $(G, \psi)$ be a $\Rep(\G)$-module homomorphism from $\cat{D}_\X$ to $\cat{D}_\Y$ satisfying $G(\Co{\X}) = \Co{\Y}$.  Then there exists a $\G$-equivariant $^*$-homomorphism $\theta$ from $\Co{\X}$ to $\Co{\Y}$ such that $\theta_\#$ is naturally isomorphic to $G$.

Furthermore, two $\Rep(\G)$-module homomorphisms $(G,\psi)$ and $(G,\psi')$ with the same underlying functor give rise to the same homomorphism $\theta$ if and only if $\psi$ and $\psi'$ are conjugate by a unitary self-equivalence of $G$.
\end{Theorem}

\begin{proof}
By Proposition~\ref{PropCatModToAlg}, we know that $\Co{\X}$ can be identified with a completion of the space $\aA = \oplus_{a\in I} \Mor(\Hsp_a\otimes \Co{\X},\Co{\X}) \otimes \Hsp_a$, and similarly for $\Co{\Y}$ as a completion of the space $\mathscr{B} = \oplus_{a\in I}\Mor(\Hsp_a\otimes \Co{\Y},\Co{\Y})$.
For any $u \in \Rep(\G)$, the action of $G$ and $\psi_{u, \Co{\X}}^{*}$ induces a linear map
\[
\Psi_u^{\Hmod{E}} \colon \Mor \left (\Hsp_u \otimes \Co{\X}, \Hmod{E} \right ) \rightarrow \Mor\left (\Hsp_u \otimes \Co{\Y}, G\Hmod{E} \right),
\]
sending $f$ to $G(f)\psi_{u,\Co{\X}}^*$. When $\Hmod{E} = \Co{\X}$, we write $\Psi_u^{\Co{\X}} = \Psi_u$, and we put $\theta = \oplus_{a \in I} \Psi_a\otimes \id_a$ as a map from $\aA$ to $\mathscr{B}$.

We first want to show that this is an algebra homomorphism. Let $f$ and $g$ be elements of $\aA$. The effect of $\theta$ on $f g$ can be expressed, using the notation from Definition~\ref{DefPenEinNotComp}, as
\begin{equation}\label{EqFonGH}
(\theta(f g))^c\otimes (\theta(f g))_c =  \lbrack G \left( f^a (\id_a \otimes g^b)((f_a \otimes g_b)^c\otimes \id_{\Co{\X}})\right)\psi_{c,\Co{\X}}^* \rbrack \otimes (f_a \otimes g_b)_c
\end{equation}
where we have dropped the associativity constraint for the module category since the latter is concrete.

By functoriality of $G$, naturality of $\psi$ and coherence of $\psi$, the morphism part in the left leg of the above formula can be written as
\[
G(f^a)\psi_{a,\Co{\X}}^*(\id_a\otimes G(g^b))(\id_a\otimes \psi_{b,\Co{\X}}^*)((f_a \otimes g_b)^c\otimes \id_{\Co{\Y}}),
\]
which can be simplified to $\Psi_a(f^a)(\id_a\otimes\Psi_b(g^b))((f_a \otimes g_b)^c\otimes \id_{\Co{\Y}})$. Since we can write $\theta(f)=\Psi_a(f^a)\otimes f_a$, we conclude that indeed $\theta(f g) = \theta(f) \theta(g)$. In the same way, the unitality of $\theta$ is proven.

Next, let us observe that $\theta$ is compatible with the involution on both algebras.  This is a consequence of the facts that $G$ `commutes' with the morphisms in $R$ and intertwines the $^*$-operations on $\cat{D}_\X$ and $\cat{D}_\Y$, and of naturality of $\psi$. Since $\theta$ is equivariant by construction, it then follows from Proposition \ref{PropComp2} that $\theta$ can be extended uniquely to an equivariant $^*$-homomorphism from $\Co{\X}$ to $\Co{\Y}$, which we denote by the same symbol.

Finally, we have to prove that $\theta_\#$ and $G$ are equivalent.  Let $\mathcal{E}$ be an object of $\cat{D}_\X$, and write $\aA_{\Hmod{E}} = \oplus_{a\in I} \Mor(u_a\otimes \Co{\X},\Hmod{E})\otimes \Hsp_a$, which we know can be identified with a dense subset of $\Hmod{E}$. Similar notation will be used for $\mathscr{B}$. Then for $f \in \aA_{\Hmod{E}}$ and $g \in \mathscr{B}$, we can define an element $n_{\Hmod{E}}(f, g)$ in $\mathscr{B}_{G\Hmod{E}}$
by
\begin{align*}
n_{\Hmod{E}}(f, g) &=\Psi_a^{\Hmod{E}}(f^a)(f_a\otimes g) \\ &= \lbrack \Psi_a^{\Hmod{E}}(f^a) (\id_a \otimes g^b)((f_a \otimes g_b)^c\otimes \id_{\Co{\Y}}) \rbrack \otimes (f_a \otimes g_b)_c.
\end{align*}

This will give a linear map $n_{\Hmod{E}}$ from the algebraic tensor product $\aA_{\Hmod{E}} \otimes \mathscr{B}$ to $\mathscr{B}_{G\Hmod{E}}$. By construction, it extends to the canonical isomorphism $\theta_\# \Co{\X} \simeq \Co{\Y} = G\Co{\X}$ at the object $\Co{\X}$. Using $\Rep(\G)$-equivariance, it then follows that $n_{\Hmod{E}}$ also extends to a unitary from $\theta_\#(\Hmod{E})$ to $G(\Hmod{E})$ for $\Hmod{E}$ of the form $u\Circt \Co{\X}$ for some representation $u$ of $\G$. By the connectedness of $\cat{D}_\X$ and linearity, we deduce that this holds for arbitrary $\Hmod{E}$. Hence $n_{\Hmod{E}}$ induces a natural unitary transformation $n\colon \theta_\# \Hmod{E} \rightarrow G\Hmod{E}$.

The way in which $n$ is constructed shows that the canonical $\psi_{\theta}$ is interchanged with $\psi$, i.e.
\[(\id\otimes \eta_{\Hmod{E}})\circ(\psi_{\theta})_{u,\Hmod{E}} = \psi_{u,\Hmod{E}}\circ \eta_{u\otimes \Hmod{E}}.\]
Indeed, taking $\xi\in \Hsp_u, f\in \mathscr{A}_{\Hmod{E}}$ and $g\in \Co{\Y}$, we have that \[\eta_{u\otimes \Hmod{E}}((\xi\otimes f)\otimes g) = G((\xi\otimes f)^a)\psi_{a,\Co{\X}}^*((\xi\otimes f)_a\otimes g).\] On the other hand,
\begin{align*}
\psi_{u,\Hmod{E}}^*(\id\otimes \eta_{\Hmod{E}})(\xi\otimes (f\otimes g)) &= \psi_{u,\Hmod{E}}^* \lbrack \xi\otimes G(f^c)\psi_{c,\Co{\X}}^*(f_c\otimes g)\rbrack \\
&= G(\id_u\otimes f^c)\psi_{u,c\otimes \Co{\X}}^*(1\otimes \psi_{c,\Co{\X}}^*)(\xi\otimes f_c\otimes g) \\
&= G(\id_u\otimes f^c)\psi_{u\otimes c,\Co{\X}}^*(\xi\otimes f_c\otimes g) \\
&= G(\id_u\otimes f^c)\psi_{u\otimes c,\Co{\X}}^*((\xi\otimes f_c)^a(\xi\otimes f_c)_a\otimes g) \\
&= G((\id_u\otimes f^c)((\xi\otimes f_c)^a\otimes \id_{\Co{\X}}))\psi_{a,\Co{\X}}^*((\xi\otimes f_c)_a\otimes g),
\end{align*}
which then reduces to the expression above.

It follows that if we have a different $\psi'$ which leads to the same $\theta$, we can construct by means of the two $n$-maps for $\psi$ and $\psi'$ a unitary self-equivalence of $G$ which conjugates $\psi$ and $\psi'$. Conversely, if $\mu$ is a natural unitary equivalence from $G$ to itself, the $\mu$-conjugated natural transformation \[\psi^\mu = (\id_u \otimes \mu_{\Co{\X}}) \psi \mu_{u \otimes \Co{\X}}^*\colon G(u \otimes \Co{\X}) \rightarrow u \otimes \Co{\Y}\] gives the same map $\Mor(u \otimes \Co{\X}, \Co{\X}) \rightarrow \Mor(u \otimes \Co{\Y}, \Co{\Y})$ as the one induced by $\psi$.
\end{proof}

\begin{Exa}\label{ExaQSubGrpHomogenEqvHoms}
Let $\mathbb{K} < \mathbb{H}$ be an inclusion of quantum subgroups of $\G$.  Then, the restriction functor $\Rep(\mathbb{H}) \rightarrow \Rep(\mathbb{K})$ is a $\Rep(\G)$-module homomorphism, and maps the trivial representation of $\mathbb{H}$ to the one of $\mathbb{K}$.  The induced $\G$-equivariant homomorphism $\CoL{\mathbb{H}\backslash\G} \rightarrow \CoL{\mathbb{K}\backslash\G}$ is the canonical inclusion of fixed point subalgebras for the respective left translation actions.
\end{Exa}

We now want to interpret Theorem \ref{ThmEquivHomFctor} in the context of bi-graded Hilbert spaces. We keep $\X$ and $\Y$ fixed quantum homogeneous spaces for $\G$. In the following,  we let $J$ (resp. $J'$) be an index set of the irreducible objects in $\cat{D}_\X$ (resp. $\cat{D}_\Y$). We denote the index corresponding to $\Co{\X}$ (resp. $\Co{\Y})$ by $\bullet$ (resp. $*$).  The $J \times J$-graded (resp. $J' \times J'$-graded) Hilbert space associated with the action of $u \in \cat{C}$ on $\cat{D}_\X$ (resp. $\cat{D}_\Y$) is denoted by $(F^\X_{r s}(u))_{r, s \in J}$ (resp. $(F^\Y_{p q}(v))_{p, q \in J'}$), and the corresponding unitaries by \[\phi_{rs,u,v}^{\cdot}\colon F_{rs}^{\cdot}(u\otimes v) \rightarrow \oplus_t F_{rt}^{\cdot}(u)\otimes F_{ts}^{\cdot}(v).\] Then if $\theta\colon \Co{\X}\rightarrow \Co{\Y}$ is an equivariant $^*$-homomorphism, we have the $J' \times J$-graded Hilbert space $\oplus_{p,r} F_{pr}$ associated with $\theta_\#$, where $F_{pr} = \Mor(x_p, \theta_\# x_r))$ for $p \in J', r \in J$, cf. Section \ref{SubSecFunctrNatTrasf}.

From Theorem \ref{ThmEquivHomFctor}, we then obtain the following corollary.

\begin{Cor}\label{CorEqvHomGrFunctrPict}
Let $\X$ and $\Y$ be quantum homogeneous spaces for $\G$.  The equivariant homomorphism from $\Co{\X}$ to $\Co{\Y}$ are in one-to-one correspondence with the classes of families of Hilbert spaces $F_{pr}$, $p\in J'$ and $r\in J$, and unitary maps
\[
\psi^u_{p r}\colon \bigoplus_{s \in J} F_{p s} \otimes F^\X_{s r}(u) \rightarrow \bigoplus_{q \in J'} F^\Y_{p q}(u) \otimes F_{q r}
\]
for $u \in \Rep(\G)$, $r \in J$, and $p \in J'$, such that $\psi^o_{p, r}$ is $\delta_{p,r}$ times the identity, $F_{p,\bullet} = \delta_{p,*}$, the diagrams
\[
\xymatrix@C=4em{
\oplus_{s} F_{p s} \otimes F^\X_{s r}(u) \ar[r]^{\psi^u_{p r}} \ar[d]_{\oplus_s \id \otimes F^\X_{s r}(T)} & \oplus_{q} F^\Y_{p q}(u) \otimes F_{q r} \ar[d]^{\oplus_q F^\Y_{p q}(T) \otimes \id}  \\
\oplus_{s} F_{p s} \otimes F^\X_{s r}(v) \ar[r]_{\psi^v_{p r}} &  \oplus_{q} F^\Y_{p q}(v) \otimes F_{q r}
}
\]
are commutative for any $T \in \Mor(u, v)$, and \[
\xymatrix@C=2em{
& \oplus_{q,t} F_{pq}^{\Y}(u)\otimes F_{qt} \otimes F_{tr}^{\X}(v)  \ar[rd]^{\oplus_q \id_{pq} \otimes \psi^v_{qr}}&  \\
\oplus_{s,t} F_{p s} \otimes F^\X_{s t}(u) \otimes F^\X_{tr}(v) \ar[ru]^{\oplus_t \psi^u_{p t}\otimes \id_{tr}} \ar[d]_{\oplus_s \id_{ps}\otimes \phi_{sr,u,v}^{\X}} &&  \oplus_{q,w} F^\Y_{p q}(u)\otimes F^\Y_{q w}(v) \otimes F_{w r}\ar[d]_{\oplus_w \phi_{pw,u,v}^{\Y}} \\ \oplus_{s} F_{p s} \otimes F^\X_{s r}(u\otimes v)  \ar[rr]_{\psi^{u\otimes v}_{pr}} &&\oplus_{w} F^\Y_{p w}(u\otimes v) \otimes F_{w r}
}
\]
is commutative.

Here two families $(F_{pr},\psi^u_{qt})$ and $(G_{pr},\mu^u_{qt})$ belong to the same class if and only if there are unitaries $U_{rs}\colon F_{pr}\rightarrow G_{pr}$ such that \[(\oplus_w (\id_{qw}\otimes U_{ws}))\psi_{qt}^{u} = \mu_{qt}^u(\oplus_s (U_{qs}\otimes \id_{st}))\] for all $q\in J'$, $t\in J$ and $u\in \Rep(\G)$.
\end{Cor}

In practice, one only needs to verify the above assumptions for all irreducible $u$ (in which case the naturality condition simplifies), or for tensor products of a $\otimes$-generating object (in which case the constraint condition simplifies). Moreover, the fact that the above Hilbert spaces are often one-dimensional in special cases makes the problem of determining the possible $\psi$ more tractable.

Important invariants of $(\X, \Y,\theta)$ are the families of integer-valued matrices $(\dim F^\X_{r s}(u_a))_{r s}$ and $(\dim F^\Y_{p q}(u_a))_{p q}$ for $a \in I$, and $(\dim(F_{p r}))_{pr}$.  These are multiplicity matrices as considered in~\cite{Was2} and~\cite{Tom1} (see Remark~\ref{RemEqvKgrp}).

Let us examine them more closely in the particular case when the larger algebra $\Co{\Y}$ is of full quantum multiplicity~\cite{BDV1,DVV1}.  This is the case if and only if $\cat{D}_\Y$ is based on a singleton $\{y\}$.  Thus, the functor $\theta_\#\colon \cat{D}_\X \rightarrow \cat{D}_\Y$ itself can be classified among the $\cC^*$-functors by the dimension of the vector spaces $F_{r} = \Mor(\theta_\#(x_r), y)$ for $r \in J$.  The next result is useful in determining the coideals inside the full quantum multiplicity ones even when there is no trace, c.f.~\cite[Corollary~4.21]{Tom1}.

\begin{Prop}\label{PropFullQMultEmbEigen}
Let $\X$ and $\Y$ be quantum homogeneous spaces over $\G$.  Assume that $\Co{\Y}$ is of full quantum multiplicity, and that there is a $\G$-equivariant homomorphism $\theta$ from $\Co{\X}$ to $\Co{\Y}$.  Then, for any $u \in \Rep(\G)$, the matrix $(\dim F^{\X}_{r s}(u))_{r, s \in J}$ has an integer-valued eigenvector for the eigenvalue $\dim F^{\Y}(u)$.
\end{Prop}

\begin{proof}
The vector $(\dim F_r)_{r \in J}$ satisfies
\begin{multline*}
\sum_{r \in J} \dim F_r \dim F^\X_{r s}(u) = \dim \Mor\left(\theta_\# ( u \otimes x_s), \Co{\Y}\right) \\
 = \dim \Mor(u \otimes \theta_\# x_s, \Co{\Y}) = \dim F^\Y(u) \dim (F_s)
\end{multline*}
for any $s \in J$ (the above sum makes sense because $(F^\X_{r s}(u))_{r, s \in J}$ is banded).  Hence it is an eigenvector of the eigenvalue $\dim F^{\Y}(u)$.
\end{proof}

\appendix
\renewcommand{\thesection}{A.\arabic{section}}
\setcounter{section}{0}

\section*{Appendix. Concrete C\texorpdfstring{$^*$}{*}-categories}\label{SecCatStuff}

In this appendix, we pick up the discussion which we started in Section~\ref{SecCCat}. It is, essentially, an elaborate write-out of the remark appearing in the proof of Theorem 2.5 of~\cite{Eti1}.

\section{Concrete semi-simple $\cC^*$-categories}

As we will show in Lemma~\ref{LemEquii}, there is essentially only one semi-simple $\cC^*$-category based on a given set $J$. This can easily be shown by using Lemma~\ref{LemEquiii}, but we would like to have a more concrete formula for the inverse of such an equivalence functor. To accomplish this, we first establish some preliminaries results. The first goal is to generalize the direct sum construction in the setting of $\cC^*$-categories.

\begin{Def}\label{DefTens}
Let $\cat{D}$ be a $\cC^*$-category. Let $X$ be an object of $\cat{D}$, and $\Hsp$ a finite-dimensional Hilbert space. An \emph{$\Hsp$-amplification of $X$} is an object $\Hsp\otimes X$ together with a linear map $\theta_X^{\Hsp}\colon \Hsp \rightarrow \Mor(X,\Hsp\otimes X)$ such that
\begin{enumerate}
\item For all $\xi,\eta\in \Hsp$, we have $\theta_X^{\Hsp}(\xi)^*\theta_X^{\Hsp}(\eta) = \langle \xi,\eta\rangle\id_X$.
\item If $\xi_i$ is an orthonormal basis of $\Hsp$, then $\sum_i \theta_X^{\Hsp}(\xi_i)\theta_X^{\Hsp}(\xi_i)^* = \mathrm{id}_{\Hsp \otimes X}$.
\end{enumerate}
\end{Def}

Note that, if $\Hsp = 0$, the second condition above implies that the $\Hsp$-amplification is a zero object.  Similarly, if $\Hsp = \C$, the $\Hsp$-amplification is equivalent to the identity functor.

\begin{Lem}
Let $\cat{D}$ be a $\cC^*$-category admitting finite direct sums, and $\mathscr{H}$ a Hilbert space of finite dimension. Then any object of $\cat{D}$ admits an $\Hsp$-amplification.  The ensuing operation $\Hf\times \cat{D}\rightarrow \cat{D}$ can be extended to an $\Hf$-module $\cC^*$-category structure on $\cat{D}$.
\end{Lem}

We recall that $\Hf$ is the category of finite-dimensional Hilbert spaces.

\begin{proof} Choose a fixed orthonormal basis $(e_i)_{i = 1}^n$ for each $\Hsp$. For an object $X\in \cat{D}$, define $\Hsp\otimes X$ as the direct sum $\oplus_{i=1}^n X$ of $n$ copies of $X$. With $v_i$ denoting the $i$-th isometric injection $X \rightarrow\oplus_i X$, the $\theta_X^{\Hsp}(\xi) = \sum_i\langle e_i,\xi\rangle v_i$ are easily seen to satisfy the conditions for an $\Hsp$-amplification. The resulting construction is obviously functorial in $X$. If $x$ is an operator $\Hsp \rightarrow \mathscr{K}$, we choose an orthonormal basis $(f_i)_i$ for $\Hsp$ and define $x\otimes \id_X$ to be the operator $\sum_i\theta_X^{\mathscr{K}}(x f_i)\theta_X^{\Hsp}(f_i)^*$ from $\Hsp\otimes X$ to $\mathscr{K}\otimes X$. Again, this is clearly independent of the chosen basis for $\Hsp$, and will give functoriality on the $\Hsp$-component. Finally, the associator for the module structure can be made as follows: given Hilbert spaces $\Hsp$ and $\mathscr{K}$ with respective bases $(f_i)$ and $(g_j)$, we define \[\phi_{\Hsp,\mathscr{K},X} =\sum_{i,j} \theta_{\mathscr{K}\otimes X}^{\Hsp}(f_i)\theta_X^{\mathscr{K}}(g_j)\theta_X^{\Hsp\otimes \mathscr{K}}(f_i\otimes g_j)^*\] as a morphism $(\Hsp\otimes \mathscr{K})\otimes X \rightarrow \Hsp\otimes (\mathscr{K}\otimes X)$.
\end{proof}

As a consequence of the $\Hf$-module structure, we obtain a natural isomorphism
\begin{equation*}
\Mor(\Hsp \otimes X, \mathscr{K} \otimes Y) \simeq \mathscr{K}\otimes \overline{\Hsp} \otimes \Mor(X, Y).
\end{equation*}
In the presentation of the right hand side, composition of morphisms involves the concatenation of the form $\Hsp \otimes \bar{\Hsp} \rightarrow \C$ `in the middle' by means of the inner product.

\begin{Not} Let $\cat{D}$ be a semi-simple $\cC^*$-category based on the set $J$. Let $r\in J$ and $X \in \cat{D}$. We denote by $X(r)$ the Hilbert space $\Mor(X_r,X)$.
\end{Not}

\begin{Lem}\label{LemEqui}
Let $\cat{D}$ be a semi-simple $\cC^*$-category based on an index set $J$. Then there is a natural unitary equivalence $X\rightarrow \oplus_{r\in J} X(r)\otimes X_r $ for $X \in \cat{D}$.
\end{Lem}

\begin{proof}
Let $Y$ be another object of $\cat{D}$.  Considering the central support of range projections for morphisms in $\Mor(Y, X)$, we see that the map
\[
\bigoplus_{r \in J}\Mor(X_r, X) \otimes \Mor(Y, X_r) \rightarrow \Mor(Y, X)
\]
induced by composition of morphisms is an isomorphism.  The left hand side of the above is, by definition of the amplification, canonically isomorphic to $\Mor\left(Y, \oplus_{r\in J} X(r) \otimes X_r\right)$.  By the Yoneda lemma, we obtain the assertion.
\end{proof}

The next definition provides the canonical semi-simple $\cC^*$-category with which we will want to compare an arbitrary one.

\begin{Def}\label{ExaCat} Let $J$ be a set. A \emph{$J$-graded Hilbert space} is a Hilbert space $\Hsp$ endowed with a direct sum decomposition $\Hsp = \oplus_{r \in J} \Hsp_r$ (the right hand side should be understood as the Hilbert space direct sum). They form a $\cC^*$-category $\mathcal{H}^J$ by considering as morphisms the grading-preserving operators,
\begin{align*}
\Mor(\Hsp, \mathscr{K}) &= \{T \in B(\Hsp,\mathscr{K})\mid \forall r\in J\colon T(\mathscr{H}_r)\subseteq \mathscr{K}_r\} \\
&= \Bigl \{ (T_r)_{r \in J} \in \prod_{r \in J} B(\Hsp_r, \mathscr{K}_r) \mid \sup_{r \in J} \left \| T_r \right \| < \infty \Bigr \}.
\end{align*}

The full subcategory of $J$-graded \emph{finite-dimensional} Hilbert spaces is denoted $\Hf^J$.
\end{Def}

The category $\Hf^J$ then forms a semi-simple $\cC^*$-category, based on the set $J$ in a natural way. Namely, an irreducible object for the label $r\in J$ is given by the graded Hilbert space $\mathbb{C}_r$ which has $\mathbb{C}$ as component at place $r$ and $0$ at the other places.

\begin{Lem}\label{LemEquii}
Let $\cat{D}$ be a semi-simple $\cC^*$-category based on a set $J$. Then the categories $\cat{D}$ and $\Hf^J$ are unitarily equivalent, an adjoint pair of equivalences being given by
\begin{align*}
X &\mapsto \bigoplus_{r \in J} X(r), &
\mathscr{H} &\mapsto \bigoplus_{r\in J} \mathscr{H}_r\otimes X_r,
\end{align*}
where $\mathscr{H}_r$ denotes the $r$-th component of $\mathscr{H}$.
\end{Lem}

\begin{proof}
An equivalence between $\cat{D}$ and $\Hf^J$ can be established by using Lemma~\ref{LemEqui} and Definition~\ref{DefTens} to define invertible unit and co-unit maps for the stated functors.
\end{proof}

\section{Functors and natural transformations}
\label{SubSecFunctrNatTrasf}

The goal of this section is to give an equally concrete description of functors between semi-simple $\cC^*$-categories, and natural transformations between them.

Let $J$ and $J'$ be index sets.  Let $\Hsp = \oplus_{p \in J', r \in J} \Hsp_{p r}$ be a Hilbert space endowed with a direct sum decomposition over the set $J' \times J$. We also assume that $\Hsp$ is \emph{column-finite} in the sense that $\sum_{p} \dim(\Hsp_{p r})$ is finite for all $r$. In particular all $\Hsp_{p r}$ are finite-dimensional. Then one has a functor $F^\Hsp$ from $\Hf^J$ to $\Hf^{J'}$ given by $(F^\Hsp \mathscr{K})_p = \oplus_r \Hsp_{p r} \otimes \mathscr{K}_r$ on objects, and $(F^\Hsp(T))_p = \oplus_r \id_{pr} \otimes T_r$ on morphisms.

If $J''$ is another index set and $\Hsp'$ is a column-finite $J'' \times J'$-graded Hilbert space, the composition of functors $F^{\Hsp'}$ and $F^\Hsp$ is given by $F^\mathscr{K}$, where the $J'' \times J$-graded Hilbert space $\mathscr{K}$ is given by the $l^{\infty}(J')$-balanced tensor product
\[
(\Hsp' \bigotimes_{l^{\infty}(J')} \Hsp)_{v r} = \bigoplus_{p \in J'} (\Hsp'_{v p}\otimes \Hsp_{p r}) \quad (v \in J'', r \in J).
\]

Let $\cat{D}$ (resp. $\cat{D}'$) be a semi-simple $\cC^*$-category based on an index set $J$ (reps. $J'$), with a system of irreducible objects $(X_r)_{r \in J}$ (resp. $(Y_p)_{p \in J'}$).  The next proposition shows that any functor between abstract semi-simple $\cC^*$-categories is induced by a column-finite $J' \times J$-graded Hilbert space as above.

\begin{Prop}\label{PropFunctSSGrHsp}
Let $F$ be a $\cC^*$-functor from $\cat{D}$ to $\cat{D}'$.  Up to the unitary equivalence of Lemma~\ref{LemEquii}, $F$ is naturally equivalent to the functor induced by the $J' \times J$-graded Hilbert space $\Hsp^F$ whose $(p, r)$-th component is $\Mor(Y_p, F(X_r))$.
\end{Prop}

\begin{proof} First of all, the graded Hilbert space $\oplus_{p,r} \Mor(Y_p,F(X_r))$ is indeed column-finite, as the $F(X_r)$ splits into a finite number of irreducible objects.

A natural equivalence as in the statement of the proposition must then be given by unitary maps
\[
\phi_p\colon \bigoplus_{r \in J} \Hsp^F_{p r} \otimes X(r) \rightarrow (F X)(p)
\]
for $p \in J'$.  On the direct summand at $r$, we define $\phi_p$ as the map
\begin{equation*}
\Mor(Y_p, F(X_r)) \otimes \Mor(X_r, X) \ni f \otimes g \mapsto F(g) \circ f \in \Mor(Y_p, F X).
\end{equation*}
Then the resulting map is indeed unitary by the semi-simplicity of $\cat{D}$.  The compatibility with the morphisms in $\cat{D}$ is apparent from the above definition of $\phi_p$.
\end{proof}

Suppose we are given two $J' \times J$-graded Hilbert spaces $\Hsp$ and $\mathscr{K}$, and an operator $T \in B(\Hsp, \mathscr{K})$ which respects the grading.  Then, we obtain a natural transformation $\eta^T$ of $F^\Hsp$ into $F^\mathscr{K}$ by the formula
\[
(\eta^T_{\mathscr{M}})_p = \bigoplus_{r \in J} T_{p r} \otimes \id_{\mathscr{M}_r}\colon F^\Hsp(\mathscr{M})_p \rightarrow F^\mathscr{K}(\mathscr{M})_p,
\]
because the norm of this operator is uniformly bounded by $\norm{T}$.  Thus, we obtain a morphism from $F^\Hsp$ to $F^\mathscr{K}$ in the category $\Fun(\Hf^J, \Hf^{J'})$ (see Remark~\ref{RemFuncCat}).

Conversely, let $F$ and $G$ be functors from $\cat{D}$ to $\cat{D}'$, and $\eta$ be a natural transformation of uniformly bounded norm from $F$ to $G$.  Then the induced maps
\[
T^\eta_{p r}\colon \Mor(Y_p, F(X_r)) \rightarrow \Mor(Y_p, G(X_r)), \quad f \mapsto \eta_{X_r} \circ f
\]
has a norm bounded from above by $\norm{\eta}$.  Now, from the way $F$ and $\Hsp^F$ is identified in Proposition~\ref{PropFunctSSGrHsp}, one sees that the above correspondences $T \mapsto \eta^T$ and $\eta \mapsto T^\eta$ are inverse to each other.  We record this for reference in the following proposition.

\begin{Prop}\label{PropNatTransSSGrTrans}
Let $F$ and $G$ be functors from $\cat{D}$ to $\cat{D}'$.  Then morphisms from $F$ to $G$ in $\Fun(\cat{D}, \cat{D}')$ can be naturally identified with grading preserving bounded operators from $\Hsp^F$ to $\Hsp^G$.
\end{Prop}

\section{Concrete semi-simple tensor $\cC^*$-categories}

We next apply the above constructions to the endomorphism tensor category $\End(\cat{D})_\textrm{f}$ associated with a semi-simple $\cC^*$-category $\cat{D}$.

\begin{Not}
Let $J$ be an index set, and denote by $\cat{E}^{J}$ the $\cC^*$-category of column-finite $J \times J$-graded Hilbert spaces $\Hsp = \oplus_{r,s\in J}\Hsp_{rs}$. As morphisms, we take the bounded operators $v\colon \Hsp\rightarrow \mathscr{K}$ which preserve the grading.
\end{Not}

By the results of Section~\ref{SubSecFunctrNatTrasf}, we can identify $\cat{E}^J$ with the tensor $\cC^*$-category of $\cC^*$-endofunctors on $\Hf^J$.  Thus, the tensor product, is given by the $l^{\infty}(J)$-balanced tensor product, and the unit object $\mathbbm{1}_{J}$ is given by $l^2(J)$ with the diagonal $J\times J$-grading $(\mathbbm{1}_{J})_{s t} = \delta_{s, t} \C$.

\begin{Lem}
The maximal rigid subcategory $\cat{E}^J_\textrm{f}$ of $\cat{E}^J$ has as its objects those $\Hsp$ which satisfy the condition \[\sup_{r}\sum_s (\dim(\Hsp_{rs}) + \dim(\Hsp_{s r}))<\infty.\] In particular, all $\mathscr{H}_{rs}$ are finite-dimensional, and only a finite number of $\Hsp_{rs}$ are non-zero on each `row' and `column', i.e.~ the grading is \emph{banded}. The dual $d(\Hsp)$ of $\Hsp$ can then be given by $d(\Hsp)_{rs} = \overline{\Hsp_{s r}} = \Hsp_{s r}^*$ with duality morphisms
\begin{align*}
R_{\Hsp}\colon l^2(J) &\rightarrow \bigoplus_{r,s\in J} \overline{\Hsp_{rs}}\otimes \Hsp_{rs}, \quad \delta_s \mapsto \sum_{r,i} \overline{\xi^{(r,s)}_i} \otimes \xi^{(r,s)}_i\\
\bar{R}_{\Hsp}\colon l^2(J) &\rightarrow \bigoplus_{r,s\in J} \Hsp_{rs}\otimes \overline{\Hsp_{rs}}, \quad \delta_r \mapsto \sum_{s,i} \xi^{(r,s)}_i \otimes  \overline{\xi^{(r,s)}_i},
\end{align*}
where the $\xi^{(r,s)}_i$ form an orthogonal basis of $\Hsp_{rs}$.
\end{Lem}

\begin{proof} The restriction on the dimensions of the $\Hsp_{rs}$ ensures that both operators $R_{\Hsp}$ and $\bar{R}_{\Hsp}$ are bounded. It is then straightforward to check that they satisfy the snake identities for a duality.

Conversely, suppose that $\oplus_{rs} \Hsp_{rs}$ admits a dual $\oplus_{rs}\mathscr{G}_{rs}$ by means of duality morphisms $(R,\bar{R})$. Then the latter decompose into maps \begin{align*} R_{rs}\colon \mathbb{C}&\rightarrow \mathscr{G}_{rs}\otimes \mathscr{H}_{sr}, & \bar{R}_{rs}\colon \mathbb{C}& \rightarrow \mathscr{H}_{rs}\otimes \mathscr{G}_{sr}.\end{align*} Let us write \begin{align*} \mathcal{J}_{rs}(\xi) = (\xi^*\otimes \id)(R_{rs}(1)) \in \Hsp_{sr}, && \mathcal{I}_{rs}(\eta) = (\eta^*\otimes \id)(\bar{R}_{rs}(1)) \in \mathscr{G}_{sr}\end{align*} for $\xi \in \mathscr{G}_{rs}$ and $\eta\in \mathscr{H}_{rs}$. Then $\mathcal{J}_{rs}$ gives an anti-linear map from $\mathscr{G}_{rs}$ to $\mathscr{H}_{sr}$, and $\mathcal{I}_{rs}$ from $\mathscr{H}_{rs}$ to $\mathscr{G}_{sr}$. The snake identities \eqref{EqDualityMors} imply that $\mathcal{I}_{rs}$ is the inverse of $\mathcal{J}_{sr}$.

By the boundedness of $R$ and $\bar{R}$, we obtain that $\sup_r \sum_s \mathrm{Tr}(\mathcal{J}_{rs}^*\mathcal{J}_{rs}) = \|R\|^2$, and similarly  $\sup_r \sum_s \mathrm{Tr}(\mathcal{I}_{rs}^*\mathcal{I}_{rs}) =\|\bar{R}\|^2$. Since $\mathcal{I}_{rs} = \mathcal{J}_{sr}^{-1}$, the trace property allows us to rewrite the latter equality as $\sup_s \sum_r \mathrm{Tr}((\mathcal{J}_{rs}^*\mathcal{J}_{rs})^{-1}) = \|\bar{R}\|^2$.

Suppose now that that the condition $\sup_{r}\sum_s (\dim(\Hsp_{rs}) + \dim(\Hsp_{s r}))<\infty$ is not satisfied. Then by symmetry we may assume that there exists a sequence $r_n$ such that $\sum_s \dim(\Hsp_{r_n,s})\geq n$. This implies that we can also find $s_n$ and a strictly positive eigenvalue $\lambda$ of $\mathcal{J}_{r_n,s_n}^*\mathcal{J}_{r_n,s_n}$ such that $\lambda \leq \frac{\|R\|^2}{n}$. But as $\lambda^{-1}\leq \|\bar{R}\|^2$, this gives a contradiction.
\end{proof}

We now show that if $\cat{D}$ is a semi-simple $\cC^*$-category based on an index set $J$, then $\End(\cat{D})_\textrm{f}$ is tensor equivalent with $\cat{E}_\textrm{f}^J$.

\begin{Prop}\label{PropConc}
Let $\cat{D}$ be a semi-simple $\cC^*$-category, based on an index set $J$. Then the categories $\End(\cat{D})_\textrm{f}$ and $\cat{E}_\textrm{f}^J$ are tensor equivalent, by means of the associations
\begin{align*}
F &\mapsto \bigoplus_{(r,s)\in J\times J} \Mor(X_r, F(X_s)), &
\Hsp &\mapsto \Bigl\lbrack X\mapsto \bigoplus_{r,s\in J} \Hsp_{rs}\otimes X(s)\otimes X_r \Bigr\rbrack.
\end{align*}
\end{Prop}

\begin{proof} We have already remarked that there are mutually inverse tensor equivalences $\End(\cat{D})\leftrightarrow \cat{E}^J$. Since equivalences preserve duality, they restrict to equivalences between $\End(\cat{D})_\textrm{f}$ and $\cat{E}_\textrm{f}^J$.
\end{proof}

\section{Module $\cC^*$-categories and bi-graded tensor functors}

This section essentially establishes that also in the categorical set-up, there is an equivalence between modules and representations. Combined with the material of the previous sections, it allows one to present a concrete and workable version of a semi-simple module $\cC^*$-category.

\begin{Lem}\label{LemEnd}
Let $\cat{C}$ be a tensor $\cC^*$-category, and $\cat{D}$ a $\cC^*$-category. Then there is an equivalence between $\cat{C}$-module $\cC^*$-category structures
$M$ on $\cat{D}$ and strong tensor $\cC^*$-functors $F\colon \cat{C}\rightarrow \End(\cat{D})$.
\end{Lem}

\begin{proof}
For module structures $M$ and tensor functors $F$, we have the associations
\begin{align*} M &\mapsto \lbrack F_M\colon U\mapsto M(U,-)\rbrack, &
F&\mapsto \lbrack M_F\colon (U,X)\mapsto F(U)(X)\rbrack,
\end{align*}
mapping all other structural morphisms in the obvious ways. These maps are clearly inverses to each other.
\end{proof}

We can now state the following useful result.

\begin{Prop}\label{PropCorrModTen}
Let $\mathcal{C}$ be a tensor $\cC^*$-category, and let $J$ be a set. Then there is an equivalence between
\begin{enumerate}
\item module $\cC^*$-structures on $J$-based semi-simple $\cC^*$-categories, and
\item strong tensor $\cC^*$-functors $\cat{C}\rightarrow \cat{E}_\textrm{f}^J$.
\end{enumerate}
Given a module $\cC^*$-category $(\cat{D},M,\phi,e)$, the corresponding tensor functor $\cat{C}\rightarrow \cat{E}_\textrm{f}^J$ is given by
\[
F\colon U \rightarrow \bigoplus_{r,s} \Mor(X_r,M(U, X_s)).
\]
Writing the right hand side above as $\oplus_{r,s} F_{rs}(U)$, the coherence maps for tensoriality are encoded as isometries
\begin{equation}\label{EqHMatStr}
F_{r s}(U) \otimes F_{s t}(V) \rightarrow F_{r t}(U \otimes V), \quad f \otimes g \mapsto \phi_{U,V,X_t}^{*} \circ (\id_U\otimes g) \circ f,
\end{equation}
\end{Prop}

\begin{proof} By Lemma~\ref{LemEnd} and Lemma~\ref{LemRig}, a $\cat{C}$-module $\cC^*$-category structure on a semi-simple $\cC^*$-category $\mathcal{D}$ based on $J$ is equivalent to giving a strong tensor $\cC^*$-functor from $\cat{C}$ to $\End(\cat{D})_\textrm{f}$. Composing with the tensor equivalence from Proposition~\ref{PropConc}, we obtain the correspondence stated in the proposition.
\end{proof}

\end{document}